\documentclass[11pt]{article}

\usepackage{latexsym}
\usepackage{amssymb}
\usepackage{amsthm}
\usepackage{amscd}

\usepackage{amsmath}

\theoremstyle{definition}
\newenvironment{conjecture}[2][Property]{\begin{trivlist}
\item[\hskip \labelsep {\bfseries #1}\hskip \labelsep {\bfseries \hspace{-0.5mm}#2.}]}{\end{trivlist}}

\newcommand{\descref}[1]{\hyperref[#1]{#1}}

\usepackage{mathrsfs}
\usepackage[all]{xy}
\usepackage{hyperref} 
\usepackage[usenames,dvipsnames]{color}
\usepackage{graphicx,eepic}
\usepackage{float}

\usepackage{color}

\theoremstyle{definition}
\newtheorem* {theorem*}{Theorem}
\newtheorem{theorem}{Theorem}[section]
\newtheorem{thmdef}[theorem]{Theorem-Definition}

\newtheorem{conj}[theorem]{Conjecture}

\theoremstyle{definition}

\newtheorem* {example*}{Example}

\newtheorem{lemma}[theorem]{Lemma}
\theoremstyle{definition}

\theoremstyle{definition}
\newtheorem* {notation}{Notation}
\newtheorem{proposition}[theorem]{Proposition}
\newtheorem{corollary}[theorem]{Corollary}

\newtheorem* {remark}{Remark}
\theoremstyle{definition}

\theoremstyle{definition}

\theoremstyle{definition}

\theoremstyle{definition}

\xyoption{dvips}

\newtheorem*{algo-p}{Algorithm for computing the polynomials $\(P_{y,w}\)_{y,w \in W}$}
\newtheorem*{algo-psig}{Algorithm for computing the polynomials $P^\sigma_{y,w}$}
\newtheorem*{algo-htilde}{Algorithm for computing the structure constants $\(  h_{x,y;z}\)_{x,y,z\in W}$}
\newtheorem*{algo-hsigma}{Algorithm for computing the structure constants $h^\sigma_{x,y;z}$}

\usepackage{fullpage}

\numberwithin{equation}{section}

\def\({\left(}
\def\){\right)}
       \newcommand{\QQ}{\mathbb{Q}}    \newcommand{\cA}{\mathcal{A}}

\def\NN{\mathbb{N}}
    \def\ZZ{\mathbb{Z}} \def\Aut{\mathrm{Aut}}  
            
  \def\wt{\widetilde}

\newcommand{\cM}{\mathcal{M}}

\def\barr{\begin{array}}
\def\earr{\end{array}}
\def\ba{\begin{aligned}}
\def\ea{\end{aligned}}
\def\be{\begin{equation}}
\def\ee{\end{equation}}

\def\qquand{\qquad\text{and}\qquad}
\def\quand{\quad\text{and}\quad}

\def\I{\mathbf{I}_*}
\def\M{\mathcal{M}}
\def\Des{\mathrm{Des}_L}
\def\omdef{\overset{\mathrm{def}}}

\def\hs{\hspace{0.5mm}}

\def\act{\ltimes}
\newcommand{\Psig}{P^\sigma}

\def\cH{\mathcal H}
\def\cM{\mathcal M}

\def\M{\mathcal M_{q^2}}

\def\A{A}
\def\a{a}

\def\ben{\begin{enumerate}}
\def\een{\end{enumerate}}

\makeatletter
\renewcommand{\@makefnmark}{\mbox{\textsuperscript{}}}
\makeatother

\allowdisplaybreaks[1]

\UseCrayolaColors

\begin{document}
\title{Positivity conjectures for Kazhdan-Lusztig theory on twisted involutions: the finite case}
\author{Eric Marberg
 \\ Department of Mathematics \\ Stanford University \\ \tt{emarberg@stanford.edu}
}
\date{}

\maketitle

\begin{abstract}
Let $(W,S)$ be any Coxeter system and let $w \mapsto w^*$ be an involution of $W$ which preserves the set of simple generators $S$. Lusztig and Vogan have shown that the corresponding set of twisted involutions (i.e., elements $w \in W$ with $w^{-1} = w^*$) naturally generates a module of the Hecke algebra of $(W,S)$ with two  distinguished bases. The  transition matrix between these bases defines  a family of polynomials $P^\sigma_{y,w}$ which one can view as a ``twisted'' analogue of the much-studied family of Kazhdan-Lusztig polynomials of $(W,S)$. The polynomials  $P^\sigma_{y,w}$ can have negative coefficients, but display several conjectural positivity properties of interest, which parallel positivity properties of the Kazhdan-Lusztig polynomials. This paper reports on some calculations which verify four such positivity conjectures in several finite cases of interest, in particular for  the non-crystallographic Coxeter systems of types $H_3$ and $H_4$. 
%We also 
%%prove some general results related to Lusztig and Vogan's constructions.
%%In particular, %we clarify the sense in  the definition of their module structure involves a canonical choice of parameters. %,  this choice has no effect on the content of our conjectures.
%%which show that the content of our conjectures is formally independent of the (non-canonical) choice of module structure on the set of twisted involutions. 
% %We also 
%  show that to check our conjectural positivity properties for all Coxeter system it suffices to check them just in the cases when $(W,S)$ is irreducible.
\end{abstract}

\setcounter{tocdepth}{2}
\tableofcontents

\section{Introduction}

\subsection{Overview} \label{overview-sect}

Let $(W,S)$ be any Coxeter system, and write $\cH_{q^2}$ for the associated \emph{Hecke algebra with parameter $q^2$}: this  is  the usual Hecke algebra (namely, a certain $\ZZ[q^{\pm 1/2}]$-algebra with a basis $\(T_w\)_{w\in W}$ indexed by $W$), but with $q$ replaced by $q^2$ in its defining relations.  (A precise definition appears in Section \ref{lv-intro}.)
%
%The \emph{$q^2$-Hecke algebra} $\H_{q^2}$ 
%
%
%
Next, fix
  % is then the free $\ZZ[q^{\pm 1/2}]$-module with basis $\{T_w : w \in W\}$, with the same multiploication as the usual Hecke algebra except that in the quadratic relation we replace $q$ by $q^2$.
%\[ T_s T_w = \begin{cases} T_{sw}, &\text{if $\ell(sw)=\ell(w)+1$}, \\ q^2T_{sw} + (q^2-1)T_w,&\text{if $\ell(sw)=\ell(w)-1$},\end{cases}\qquad\text{for }s \in S,\ w\in W.\]
an automorphism  $* : W \to W$ with order one or two which preserves the set of simple generators   $S$. Write $\I$ for the corresponding set of twisted involutions (i.e., elements $w \in W$ with $w^{-1} = w^*$), 
% = \{ w \in W : w^{-1} = w^*\}\] for the corresponding set of {twisted involutions},
 and let 
$\cM_{q^2}$ be the free $\ZZ[q^{\pm 1/2}]$-module which this set generates.

Lusztig and Vogan \cite{LV2, LV1} have shown that 
$\cM_{q^2}$ naturally carries a nontrivial $\cH_{q^2}$-module structure, which gives rise to a distinguished basis of $\cM_{q^2}$ sharing many formal properties with the much-studied \emph{Kahzdan-Lusztig basis} of $\cH_{q^2}$.
In particular, the transition matrix from the standard basis of $\cM_{q^2}$ to its distinguished basis defines a family of ``twisted Kazhdan-Lusztig polynomials'' $(P^\sigma_{y,w})_{y,w \in \I} \subset \ZZ[q]$, formally similar to the \emph{Kazhdan-Lusztig polynomials} $(P_{y,w})_{y,w \in W} \subset \ZZ[q]$ attached to $(W,S)$.  The details   of these constructions are given in Sections \ref{kl-intro} and \ref{lv-intro}.

 Several remarkable properties of the Kazhdan-Lusztig basis of $\cH_{q^2}$ appear to have ``twisted'' analogues for the module $\cM_{q^2}$.
For example, one of the most famous aspects of the original Kazhdan-Lusztig polynomials $(P_{y,w})_{y,w \in W}$, only recently proved in complete generality by Elias and Williamson \cite{EWpaper}, is that their   coefficients are always  nonnegative.
%Kazhdan and Lusztig \cite{KL} proved that this holds when $W$ is crystallographic using intersection cohomology some three decades ago, and  Elias and Williamson \cite{EWpaper} have only recently discovered an algebraic proof for all Coxeter groups.
The twisted Kazhdan-Lusztig polynomials $(P^\sigma_{y,w})_{y,w \in \I}$ can have negative coefficients, but  Lusztig and Vogan \cite{LV2}   have shown by  geometric arguments   that   the modified polynomials $ \frac{1}{2} (P_{y,w} \pm  P^\sigma_{y,w})$ for $y,w \in \I$  have nonnegative  coefficients whenever $(W,S)$ is crystallographic. In fact, for any choice of $(W,S)$ and $*$,  the polynomials $\frac{1}{2} (P_{y,w} \pm  P^\sigma_{y,w})$ belong to $\ZZ[q]$ by \cite[Theorem 9.10]{LV2}, and Lusztig \cite[Conjecture 9.12]{LV2} has conjectured that their coefficients are  always  nonnegative. 

Section \ref{conj-sect} presents three additional conjectural positivity properties related to the ``Kazhdan-Lusztig basis'' of the $\cH_{q^2}$-module $\cM_{q^2}$.
%All but the last of these four properties appeared in the precedent work \cite{EM1}, where they were proved for \emph{universal} Coxeter systems (that is, Coxeter systems $(W,S)$ for which the product $st$ has order 1 or $\infty$ for each $s,t \in S$).
We prove that these positivity properties hold for arbitrary Coxeter systems if they hold for irreducible Coxeter systems, provided that analogous positivity conjectures related to the ordinary Kazhdan-Lusztig polynomials hold.
In addition, we report on some calculations performed using extensions \cite{MyCode} to du Cloux's program {\tt Coxeter} \cite{Coxeter}, which verify our four positivity properties in several finite cases (in particular, for the non-crystallographic Coxeter systems of types $H_3$ and $H_4$).
A more  detailed summary of our results  appears in Section \ref{summary-sect} at the end of this introduction, after some minimal preliminaries in Sections \ref{kl-intro}, \ref{lv-intro}, and \ref{conj-sect}.

\subsection{Kazhdan-Lusztig theory}\label{kl-intro}

 Throughout we write $\ZZ$ for the integers and $\NN = \{0,1,2,\dots\}$ for the nonnegative integers, and we adopt the following conventions:

\begin{itemize}
\item Let $(W,S)$ be a Coxeter system with length function $\ell : W \to \NN$.

%\item Given $x,y \in W$, we define $\ell(x,y) = \ell(y) - \ell(x)$.

\item Let $\leq $ denote the Bruhat order on $W$.  %Recall that in this partial order,  $y \leq w$  if and only if for each reduced expression $w = s_1\cdots s_k$ with $s_i \in S$, there are indices $1\leq i_1 < \dots < i_m \leq k$ such that  $y = s_{i_1}\cdots s_{i_m}$.

% so that if $y,w \in W$ then $y \leq w$  if and only if for each reduced expression $(s_1,\dots, s_k)$ for $w$ there are indices $1\leq i_1 \leq \dots \leq i_m \leq k$ such that $y = s_{i_1}\cdots s_{i_m}$.

\item Let $\cA = \ZZ[v,v^{-1}]$ be the ring of Laurent polynomials over $\ZZ$ in an indeterminate $v$.

\item Let $q =v^2$.  In the sequel, we will refer to $v$ in place of the parameter $q^{1/2}$   in Section \ref{overview-sect}.

\end{itemize}
%Thus $W$ is a group; $S\subset W$ is a set of elements of order two which generate $W$; $\ell(w)$ is the minimum integer $k$ such that $w = s_1s_2\cdots s_k$ for some $s_i \in S$; and
%we have $y \leq w$ for two elements $y,w \in W$ if and only if whenever $w = s_1s_2\cdots s_{\ell(w)}$ for some $s_i \in S$, there are indices $1\leq i_1 \leq i_2 \leq \dots \leq i_m \leq \ell(w)$ such that $y = s_{i_1}s_{i_2}\cdots s_{i_m}$. In particular, if $y <w $ then $\ell(y) < \ell(w)$.
%Finally, the ring $\cA$ will now occupy the role which $\ZZ[q^{\pm1/2}]$ played in the previous section. 
For background on Coxeter systems and the Bruhat order, see for example \cite{CCG,Hu,Lu}.
%; see also \cite{H1,H2,H3} for a discussion of the special properties of the poset structure on $\I$ induced by the Bruhat order.

%Here we briefly review three longstanding positivity conjectures related to the Hecke algebra of a Coxeter system. These will motivate three analogous conjectures in the next section.
Here we briefly recall the definition of the Kazhdan-Lusztig polynomials attached to $(W,S)$.
Let $\cH_q$ denote the free $\cA$-module with basis $\{ t_w : w \in W\}$. 
This module has a unique $\cA$-algebra structure  with respect to which the multiplication rule
\[ t_s t_w = \begin{cases} t_{sw} & \text{if }sw>w \\ q t_{sw} + (q-1) t_w &\text{if }sw<w \end{cases}
\] holds 
for each $s \in S$ and $w \in W$. 
We remark that  the element $t_w \in \cH_q$ is  more often denoted in the literature by the symbol $T_w$, but here we reserve the latter notation for the   Hecke algebra $\cH_{q^2}$, to be introduced in the next section.

We refer to the algebra $\cH_q$ as the \emph{Hecke algebra of $(W,S)$ with parameter $q$.} Standard references  for this much-studied object include, for example, \cite{CCG,Hu,KL,Lu}.
The Hecke algebra possesses a unique ring involution  $\overline{\ } : \cH_q \to \cH_q$ with 
$\overline{v^n} = v^{-n}$ and $\overline {t_w} = \(t_{w^{-1}}\)^{-1}$ 
%\[ \overline{ v^n} = v^{-n} \qquand \overline{   t_w} = \(t_{w^{-1}}\)^{-1}\] for 
all $n \in \ZZ$ and $w \in W$, referred to as the \emph{bar operator}, and this gives rise to the following  theorem-definition from Kazhdan and Lusztig's seminal paper \cite{KL}.% of the Kazhdan-Lusztig elements $c_w \in \cH_q$.

\begin{thmdef}[Kazhdan and Lusztig \cite{KL}]  \label{kl-thmdef} For each $w \in W$ there is a unique family of polynomials $\( P_{y,w} \)_{y \in W} \subset \ZZ[q]$ 
with the following three properties:
\ben
\item[(a)] The element $ c_w \omdef =  v^{-\ell(w)} \cdot \sum_{y \in W}   P_{y,w}\cdot t_y \in \cH_q$
has $\overline{c_w} = c_w$.
\item[(b)] $P_{y,w} = \delta_{y,w}$ if $y \not < w$ in the Bruhat order.
\item[(c)] $P_{y,w}$ has degree at most $ \frac{1}{2} \( \ell(w)-\ell(y)-1\)$  as a polynomial in $q$ whenever $y < w$.
\een
\end{thmdef}

\begin{remark} Here and in the sequel,  the Kronecker delta $\delta_{y,w}$ has the usual meaning of   $\delta_{y,w} =1$ if $y=w$ and $\delta_{y,w} = 0$ otherwise. 
\end{remark}

%In part (b)  the Kronecker delta $\delta_{y,w}$ has the usual meaning:   $\delta_{y,w} =1$ if $y,w$ and $\delta_{y,w} = 0$ otherwise. 
The polynomials $(P_{y,w})_{y,w \in W}$ are the \emph{Kazhdan-Lusztig polynomials} of the Coxeter system $(W,S)$. Property (b) implies that the  
elements $\( c_w \)_{ w \in W}$ form an $\cA$-basis for  $\cH_q$, which one calls the \emph{Kazhdan-Lusztig basis}. %\qquad\text{for }x,y,z \in W.\]
%
%
%Given $y,w \in W$, we define 
%\[ \mu(y,w)=\text{ the coefficient of $v^{\ell(w)-\ell(y)-1}$ in the polynomial $P_{y,w} \in \ZZ[v^2]$.}\] 
We note the following well-known multiplication formula for use later.

\begin{theorem}[Kazhdan and Lusztig \cite{KL}] \label{kl-thm} Let $w \in W$ and $s \in S$. Then $c_s = v^{-1}(t_s+1)$ and
\[ c_s c_w = \begin{cases} (v+v^{-1}) c_w&\text{if $sw<w$} \\[-10pt]\\
 c_{sw} + \sum_{{z \in W ;\hs sz<z<w}} \mu(z,w) c_z&\text{if $sw>w$}\end{cases}\]
 where 
$ \mu(z,w)$ denotes the coefficient of $v^{\ell(w)-\ell(z)-1}$ in the polynomial $P_{z,w} \in \ZZ[v^2]$.
 \end{theorem}

%\subsection{Twisted Kazhdan-Lusztig theory}\label{twisted-intro}

%The present work focuses not on Conjectures \descref{A}, \descref{B}, and \descref{C} but rather on  a set of analogous conjectures for a certain Hecke algebra module, which we  describe in  the present section.

%Above we defined an $\cA$-algebra $\cH_q$ possessing two $\cA$-bases indexed by $W$: the standard basis $\( t_w\)_{w \in W}$ and the Kazhdan-Lusztig basis $\(c_w\)_{w \in W}$. Here 

\subsection{Twisted Kazhdan-Lusztig theory}\label{lv-intro}

Following \cite{LV2,LV1}, we now introduce a slightly different Hecke algebra $\cH_{q^2}$, possessing an analogous pair of $W$-indexed $\cA$-bases
 which we will  denote using capital letters by  $\(T_w\)_{w \in W}$ and  $\(C_w\)_{w \in W}$. %Informally $\cH_{q^2}$ is given by just replacing $q$ by $q^2$ in the defining relations of $\cH_q$. The use of the parameter $q^2$ is not essential to what follows, but simplifies several statements and makes the ``twisted'' theory we are about to describe more closely parallel the constructions in the previous section.
Explicitly,
 let $\cH_{q^2}$ denote the free $\cA$-module with basis $\{ T_w : w \in W\}$. 
This module has a unique $\cA$-algebra structure  with respect to the slightly altered multiplication rule
\[ T_s T_w = \begin{cases} T_{sw} & \text{if }sw>w \\ q^2 T_{sw} + (q^2-1) T_w &\text{if }sw<w \end{cases}
\] holds 
for each $s \in S$ and $w \in W$.  We refer to $\cH_{q^2}$ with this structure as the \emph{Hecke algebra of $(W,S)$ with parameter $q^2$.}
This algebra likewise possesses a unique ring involution  $\overline{\ } : \cH_{q^2} \to \cH_{q^2}$ with 
$\overline{v^n} = v^{-n}$ and $\overline {T_w} = \(T_{w^{-1}}\)^{-1}$
%\[\overline{ v^n} = v^{-n} \qquand \overline{   T_w} = \(T_{w^{-1}}\)^{-1}\] 
for all $n \in \ZZ$ and $w \in W$, which fixes each of the elements
 \[C_w \omdef = q^{-\ell(w)} \cdot \sum_{y \in W} P_{y,w}(q^2) \cdot T_y \in \cH_{q^2}\qquad\text{for }w \in W.\] The elements  $\( C_w \)_{ w \in W}$ form an $\cA$-basis of $\cH_{q^2}$ which one refers to as its \emph{Kazhdan-Lusztig basis}.

The construction which is the main topic of this work is now given as follows.
Choose an automorphism $w \mapsto w^*$ of $W$ with order $\leq 2$ such that $s^* \in S$ for each $s \in S$,
and
write $\I$ for the corresponding set of \emph{twisted involutions}
\[\I = \{ w \in W : w^* = w^{-1}\}.\]
%Lusztig and Vogan \cite{LV2,LV1} have shown that the $\cA$-module generated by $\I$ has a certain nontrivial $\cH_{q^2}$-module structure, with its own notion of a ``bar operator,'' and with a distinguished bar invariant $\cA$-basis sharing many formal properties with the Kazhdan-Lusztig basis of $\cH_q$. %This  distinguished basis is our central object of study.
Lusztig and Vogan's paper \cite{LV1} first established the following trio of Theorem-Definitions in the  case that $W$ is a Weyl group or affine Weyl group and $*$ is trivial; Lusztig's paper \cite{LV2} then extended these results to arbitrary Coxeter systems. %, this result appeared first in Lusztig and Vogan's paper \cite{LV1}.

\begin{notation}
Given $s \in S$ and $ w\in \I$,   let $s \act w$ denote the unique element in the intersection of $\{ sw,sws^*\}$ with $\I\setminus \{w\}$. Explicitly, $s\act w = sw = ws^*$ if $w=sws^*$ and $s\act w = sws^*$ otherwise. 
%Note that while $s \act (s \act w) = w$, the operation $\act : S \times \I\to \I$ generally does not extend to a group action of $W$ on $\I$.
\end{notation}

\begin{thmdef}[Lusztig and Vogan \cite{LV1}; Lusztig \cite{LV2}]\label{lv-thmdef1}
Let $\M$ be the free $\cA$-module with basis $\{ \a_w : w \in \I\}$.  Then  $\M$ has a unique $\cH_{q^2}$-module structure with respect to which the following multiplication rule holds for each  $s \in S$ and $w \in \I$:
\[ T_s \a_w =
\left\{
\ba 
a_{s\act w} \ {\color{white}+}\ &  		&&\quad\text{if $s\act w =sws^* > w$} \\
(q+1)a_{s\act w} \ +\ & qa_w			&&\quad\text{if $s\act w =sw > w$} \\
(q^2-q)a_{s\act w} \ +\ &(q^2-q-1) a_w 	&&\quad\text{if $s\act w =sw < w$} \\
q^2a_{s\act w} \ +\  & (q^2-1)a_w 		&&\quad\text{if $s\act w =sws^* < w$.}
\ea
\right.
\]
\end{thmdef}

\begin{thmdef}[Lusztig and Vogan \cite{LV1}; Lusztig \cite{LV2}]\label{lv-thmdef2}
 There is a unique $\ZZ$-linear involution $\overline{\ } : \M \to \M$ such that $\overline{a_1} = a_1$ and $\overline{h\cdot m} = \overline h \cdot \overline m$ for all $h \in \cH_{q^2}$ and $m \in \cM_{q^2}$.
%This bar operator acts on the standard basis of $\cM_{q^2}$ by the  formula $\overline{\a_w} =(-1)^{\ell(w)}\cdot  (T_{w^{-1}})^{-1}\cdot \a_{w^{-1}}$ for $w \in \I$. %, where $\sgn(w) = (-1)^{\ell(w)}$. 
\end{thmdef}

Lusztig \cite{LV2} has shown moreover  that the bar operator just defined acts on the standard basis of $\cM_{q^2}$ by the  formula $\overline{\a_w} =(-1)^{\ell(w)}\cdot  (T_{w^{-1}})^{-1}\cdot \a_{w^{-1}}$ for $w \in \I$.

\begin{thmdef}[Lusztig and Vogan \cite{LV1}; Lusztig \cite{LV2}]\label{lv-thmdef3} For each $w \in \I$ there is a unique family of polynomials $\( \Psig_{y,w} \)_{y \in \I} \subset \ZZ[q]$ 
with the following three properties:
\ben
\item[(a)] The element $ %\label{A_w-def} 
A_w \omdef=  v^{-\ell(w)} \cdot \sum_{y \in \I}   \Psig_{y,w}\cdot a_y\in \M$
has $\overline{A_w} = A_w$.
\item[(b)] $\Psig_{y,w} = \delta_{y,w}$ if $y \not < w$ in the Bruhat order.
\item[(c)] $\Psig_{y,w}$ has degree at most $ \frac{1}{2} \( \ell(w)-\ell(y)-1\)$  as a polynomial in $q$ whenever $y < w$.
\een
\end{thmdef}

The elements $\( A_w \)_{ w \in \I}$ form an $\cA$-basis for the module $\cM_{q^2}$, which we sometimes refer to this as the ``twisted Kazhdan-Lusztig basis.'' 
%of $\cM_{q^2}$.
Likewise, we call the polynomials $\Psig_{y,w}$ the \emph{twisted Kazhdan-Lusztig polynomials} of the triple $(W,S,*)$.

\subsection{Positivity properties}
\label{conj-sect} 

The primary results of this paper concern four conjectural positivity properties of the twisted Kazhdan-Lusztig polynomials $\Psig_{y,w}$. These 
are patterned on the following (partially conjectural) properties of the Kazhdan-Lusztig polynomials $P_{y,w}$ of an arbitrary Coxeter system $(W,S)$. 

\begin{conjecture}{A}\label{A}
The polynomials $P_{y,w}$ have nonnegative  coefficients for all $y,w \in W$.
\end{conjecture}

\begin{conjecture}{B}\label{B}
 The   polynomials $P_{y,w}$ are decreasing for fixed $w$, in the sense that the difference $P_{y,w} - P_{z,w}$ has nonnegative  coefficients whenever $y,z,w \in W$ and $y \leq z$.
\end{conjecture}

%Denote the structure coefficients of $\cH_q$ in the Kazhdan-Lusztig basis by $\( h_{x,y;z} \)_{x,y,z \in W}$; i.e., these are the Laurent polynomials in $ \cA$ satisfying $c_x c_y = \sum_{z \in W} h_{x,y;z} c_z $ for $x,y,z \in W$.

Let $\( h_{x,y;z} \)_{x,y,z \in W}$ denote  the structure constants of $\cH_q$ in the Kazhdan-Lusztig basis, i.e.,  the Laurent polynomials in $ \cA$ satisfying 
$c_x c_y = \sum_{z \in W} h_{x,y;z} c_z $ for $x,y,z \in W$.

\begin{conjecture}{C}\label{C}
The Laurent polynomials $h_{x,y;z}$ have nonnegative coefficients for all $x,y,z \in W$.
\end{conjecture}

\begin{conjecture}{D}\label{D}
For any $x,y,z \in W$, if the Laurent polynomial $h_{x,y;z}$ has degree $d$ in $v$ (where we consider $0$ to have degree $0$), then $v^d h_{x,y;z}$ is a unimodal polynomial in $q=v^2$.
\end{conjecture}

\begin{remark}
Recall that a polynomial $a_0 + a_1 x + a_2x^2 + \dots + a_n x^n \in \ZZ[x]$, where each $a_i \in \ZZ$, is \emph{unimodal} if 
$ 0 \leq a_0 \leq \dots\leq a_{i-1} \leq a_i \geq a_{i+1} \geq \dots \geq a_n \geq 0$ for some index $i$. 
%Following \cite{zel}, we define the \emph{darga} of $f$ to be the sum $m+n$, where $m$ is the smallest index such that $a_m \neq 0$. Finally, we say that $f$  is \emph{symmetric} if $a_i = a_{N-i}$ for every $i$, where $N = darga(f)$.
%(The zero polynomial is thus unimodal, and all unimodal polynomials have nonnegative coefficients.)
%We adopt the convention in this work that the degree of the zero polynomial is $-1$. 
%This removes any potential confusion about references to the product of  an arbitrary polynomial $f\in \cA$ by $v$ to the power of its degree.
%Finally, we should remark that 
The content of Property \descref{D} really only concerns unimodality, for it always holds   that   $v^dh_{x,y;z}$  is   a polynomial in $q$ if $d$ is the degree of $h_{x,y;z}$ as a Laurent polynomial in $v$; see Corollary \ref{vv-cor}.
\end{remark}

Naturally accompanying the preceding properties is this conjecture.

\begin{conj} Properties A, B, C, and D hold for all Coxeter systems $(W,S)$.
\end{conj}

%The  structure constants of the Kazhdan-Lusztig basis have the property that if $h_{x,y;z}$ has degree $d$ in $v$, then $v^dh_{x,y;z}$  is   a polynomial in $v$ of degree $2d$, which is symmetric in the sense that its coefficients of degrees $i$ and $2d-i$ coincide for every $i$; see Corollary \ref{vv-cor}. More appears to be true:

Elias and Williamson's recent proof of Soergel's conjecture \cite{EWpaper} shows that at least Properties \descref{A} and \descref{C}  hold for all Coxeter systems. 
Property \descref{B} is known to hold for all finite Coxeter systems by work of Irving \cite{Ir} and du Cloux \cite{Fokko}. 
Property \descref{D} has received the least attention in the literature. 
%It is expected that this statement can be proved using the  methods introduced in \cite{EWpaper}; 
Computations of du Cloux \cite{Fokko} at least show that this property holds for dihedral Coxeter systems and in types $H_3$ and $H_4$. Using du Cloux's program {\tt Coxeter} \cite{Coxeter} we have in turn  checked (appealing to Proposition \ref{reduction-prop}  below) that Property \descref{D} holds 
for all finite Coxeter systems whose irreducible factors have rank at most five. (Recall that the \emph{rank} of $(W,S)$ is the size of $S$.)

%Concerning this last conjecture, we remark that if $h_{x,y;z}$ has degree $d$ in $v$, then $v^dh_{x,y;z}$  is always  a polynomial in $q$ of degree $d$, which is symmetric in the sense that its coefficients of degrees $i$ and $d-i$ coincide for every $i$; see Corollary \ref{vv-cor}.

%The parallels between Theorem-Definitions \ref{kl-thmdef} and \ref{twisted-thmdef} suggest some obvious analogues of the conjectures in the previous section, but these statements turn out not  to be the right ones.
The appropriate ``twisted'' analogues of the preceding properties are not the obvious ones suggested by the formal parallels between Theorem-Definitions \ref{kl-thmdef} and \ref{lv-thmdef3}. 
%(Those statements, that $\(\Psig_{y,w}\)_{y,w \in \I} \subset \NN[q]$, etc., turn out to be false.)
Instead we proceed as follows.
Define $P^+_{y,w}, P^-_{y,w} \in \QQ[q]$ by \[P^\pm_{y,w} = \tfrac{1}{2} \( P_{y,w} \pm \Psig_{y,w}\)\qquad\text{for each }y,w \in \I.\] While  the polynomials $\Psig_{y,w}$ may have negative coefficients, Lusztig proves that the polynomials $P^\pm_{y,w}$ actually have integer coefficients  \cite[Theorem 9.10]{LV2}.
Consider the following conjectural properties related to these polynomials:

\begin{conjecture}{A$'$}\label{A$'$}
Both $P^+_{y,w}$ and $P^-_{y,w}$ have nonnegative  coefficients for all $y,w \in \I$.
\end{conjecture}

%This statement refines Conjecture \descref{A} in the sense that $P^+_{y,w} + P^-_{y,w} = P_{y,w}$ for $y,w \in \I$. We  introduced in \cite{EM1} the following stronger conjecture, which is likewise a refinement of Conjecture \descref{B}.

\begin{conjecture}{B$'$}\label{B$'$} The   polynomials $P^\pm_{y,w}$ are decreasing for fixed $w$, in the sense that the differences
$P^+_{y,w} - P^+_{z,w}$ and  $P^-_{y,w} - P^-_{z,w}$ have nonnegative  coefficients whenever $y,z,w \in \I$ and $y \leq z$.
\end{conjecture}

To give analogues of Properties \descref{C} and \descref{D}, for each $x \in W$ and $y \in \I$ define $\bigl( \wt h_{x,y;z}\bigr)_{z \in W}\subset  \cA$ and $ \bigl(h^\sigma_{x,y;z}\bigr)_{z \in \I} \subset  \cA$ as the Laurent polynomials satisfying
\be\label{h-def} c_x c_y c_{{x^*}^{-1}} = \sum_{z \in W} \wt h_{x,y;z} c_z
\qquand
C_x \A_y = \sum_{z \in\I} h^\sigma_{x,y;z}\A_z.
\ee
Note  that $c_x,c_y,c_z \in \cH_q$  while $C_x \in \cH_{q^2}$ and $A_y \in \cM_{q^2}$. Now, %given $x \in W$ and $y,z \in \I$, 
 define $h^+_{x,y;z}, h^-_{x,y;z} \in \QQ[v,v^{-1}]$ by \be\label{hpm-def}
 h^\pm_{x,y;z}  = \tfrac{1}{2}\( \wt h_{x,y;z} \pm  h^\sigma_{x,y;z}\)
 \qquad\text{for each $x \in W$ and $y,z \in \I$.}
 \ee Though not clear \emph{a priori}, these Laurent polynomials too always  have integer coefficients \cite[Proposition 2.11]{EM1}. % and we have the following analogues of Conjectures \descref{C} and \descref{D}.
 
\begin{conjecture}{C$'$}\label{C$'$}
Both $h^+_{x,y;z}$ and $h^-_{x,y;z}$ have nonnegative  coefficients for all $x \in W$ and $y,z \in \I$.
\end{conjecture}

%Like the structure constants $h_{x,y;z}$, the Laurent polynomials $h^\pm_{x,y;z}$ also belong to the subring $\ZZ[v+v^{-1}] \subset \cA$ (see Proposition \ref{rest-vv}), and we have this analog of Conjecture \descref{D}, to which no counter-examples are yet known:

\begin{conjecture}{D$'$}\label{D$'$}
For any $x \in W$ and $y,z \in \I$, if the Laurent polynomials $h^+_{x,y;z}$ and $h^-_{x,y;z}$ have degrees $d_+$ and $d_-$ in $v$, then $v^{d_+} h^+_{x,y;z}$ and $v^{d_-} h^-_{x,y;z}$ are  unimodal polynomials in $q=v^2$.
\end{conjecture}

The main purpose of this paper is to provide  evidence to support the following conjecture:

\begin{conj} Properties A$'$, B$'$, C$'$, and D$'$ hold for all triples $(W,S,*)$ where $(W,S)$ is a Coxeter system and $* \in \Aut(W)$ is an $S$-preserving involution.
\end{conj}

%The  status of these conjectures is less definitive than that of their   predecessors.
Lusztig and Vogan  have shown that Property \descref{A$'$} holds when $W$ is a Weyl group or affine Weyl group; see \cite[\S3.2 and \S7]{LV1}. In these cases, \cite[Section 5]{LV1} also 
mentions without proof that Property \descref{C$'$}  holds (when $*$ is trivial). The companion work \cite{EM1} establishes  Properties \descref{A$'$}, \descref{B$'$}, and \descref{C$'$}  in the case that  $(W,S)$ is a {universal} Coxeter system. %, i.e., one  in which the order of the product $st$ is infinite for any distinct $s,t \in S$. 
We did not  consider Property \descref{D$'$} in \cite{EM1}, but one can likely adapt the arguments in \cite{EM1} to also prove this fourth property in the universal case.
%The next section summarizes our main results continuing this progress.

\subsection{Outline of main results}
\label{summary-sect}

%The rest of this paper is divided into three sections.
%Section \ref{unique-sect} examines  the question of whether one is justified in considering the  module structure \eqref{module-def} to be ``canonical.''
%It turns out that one can replace the parameters in \eqref{module-def} in a number  of different ways and still be left with a valid $\cH_{q^2}$-module structure on $\cM_{q^2}$; see Theorem \ref{unique-thm}. 
%Nevertheless there is a sense in which the choice of parameters made in the statement of Theorem-Definition \ref{lv-thmdef} above \emph{is} canonical; see Corollary \ref{unique-cor}. . . 
%
%
Section \ref{recover-sect} reviews several formulas concerning the module $\cM_{q^2}$ from Lusztig's paper \cite{LV2}. Notably, Corollary \ref{main-recurrence} gives a recurrence for the polynomials $\Psig_{y,w}$. In Section \ref{recurrence-sect} we discuss how this recurrence leads to an algorithm capable of verifying our positivity properties in finite cases.

In Section \ref{reduction-sect} we show that the
eight properties in Section \ref{conj-sect} hold for all Coxeter systems if they hold for all \emph{irreducible} Coxeter systems. (For the definition of irreducibility, see \cite[Section 6.1]{Hu}.)
%Recall that the \emph{Coxeter diagram} of $(W,S)$ is the weighted graph with vertex set $S$ which has an edge between $s,t \in S$ if and only if $st\neq ts$. 
%The weight of the edge between $s,t \in S$ is then  the order of the product $st \in W$.
%If $S'\subset S$ is   the (nonempty) set of vertices of a connected component of this graph and $W' = \langle S'\rangle $ is the subgroup of $W$ which $S'$ generates, then $(W',S')$ is a Coxeter system which one calls an \emph{irreducible factor} of $(W,S)$. A Coxeter system is irreducible if it has exactly one irreducible factor.
Table \ref{types-tbl}  enumerates all triples $(W,S,*)$  where $(W,S)$ is an irreducible finite Coxeter system and $*\in \Aut(W)$ is an $S$-preserving involution.
In more detail,
let X be one of the letters A, B, C, or D, so that our positivity properties may each be referred to  as either Property X or Property X$'$.
We adopt the following convention:   whenever we say that Property X$'$ holds for a Coxeter system $(W,S)$, we 
mean that the Property X$'$ holds for $(W,S)$ with respect to all  choices of $S$-preserving involution $* \in \Aut(W)$.
Propositions \ref{reduction-prop} and \ref{reductionprime-prop} together imply the following.

\begin{theorem}\label{x-thm} Let $\text{X} \in \{ \text{A}, \text{B},\text{C},\text{D}\}$ and let $(W,S)$ be a Coxeter system. If Properties X and X$'$ hold for all irreducible factors of $(W,S)$, then Properties X and X$'$ hold for $(W,S)$.
\end{theorem}

%With this statement in hand, we may finally  report on our progress towards verifying our conjectures for finite Coxeter systems.
 %Lusztig and Vogan have already shown that Conjecture \descref{A$'$}  holds whenever $(W,S)$ is a Weyl group (see \cite[\S3.2 and \S7]{LV1}).
 In Section \ref{dihedral-sect} we prove that Properties \descref{A$'$} and \descref{B$'$} hold for all Coxeter systems of rank two. %(Recall that the \emph{rank} of $(W,S)$ is   the cardinality of $S$.) 
 Our further results are computational in nature. We have obtained these from extensions \cite{MyCode} we have written
 to the final version du Cloux's C++ program {\tt Coxeter} \cite{Coxeter}. Our extended version of {\tt Coxeter} allows a user to compute the polynomials $\Psig_{y,w}$, $\wt h_{x,y;z}$, $h^\sigma_{x,y;z}$, and $h^\pm_{x,y;z}$ for a given finite Coxeter system with involution.
%One can access these extensions to {\tt Coxeter}  online \cite{MyCode}.
%; this code includes detailed  comments describing our implementation, which is outlined in Section \ref{recurrence-sect}. 

  Using the extended program,
  we have been able to check directly that Properties \descref{A$'$} and \descref{B$'$} hold for each of the triples $(W,S,*)$ listed in Table \ref{maxcoeff-P-tbl} and that Properties \descref{C$'$} and \descref{D$'$} hold for the triples listed in Table \ref{maxcoeff-h-tbl}.
Of the cases considered, type $H_4$ is by far the most computationally intensive, requiring for the calculation of the polynomials $\(h^\pm_{x,y;z}\)_{x\in W;\hs y,z \in \I}$  around 10 days' computing time (on a 2.26 GHz MacBook Pro with 4 GB of main memory) and around 93 GB of memory to store all uncompressed output files.  Even in this type, verifying Properties \descref{A$'$} and \descref{B$'$} only takes a few minutes, however. 
%Tables \ref{maxcoeff-P-tbl} abd  \ref{maxcoeff-h-tbl} show the greatest nonzero coefficients of the polynomials thus computed, for the finite Coxeter systems with involution where our algorithms are effective.
%
Combining this discussion with the results of \cite{LV1} and with Theorem \ref{x-thm}, we 
 arrive at the following statement.

\begin{theorem}\label{last-thm} Let $(W,S)$ be a   Coxeter system with an $S$-preserving involution $* \in \Aut(W)$. Properties \descref{A$'$}, \descref{B$'$}, \descref{C$'$}, \descref{D$'$} then hold in at least the following cases:
\ben
\item[(a)] Property \descref{A$'$} holds whenever $(W,S)$ is finite.
\item[(b)] Property \descref{B$'$} holds if all irreducible factors of $(W,S)$ are finite with rank at most 6.
\item[(c)] Properties \descref{C$'$}, \descref{D$'$} hold if all irreducible factors of $(W,S)$ are finite with rank 1, 3, 4, or 5.
\een
\end{theorem}

%\begin{remark}
Our calculations actually show a little  more than this  statement indicates. Specifically, Property \descref{B$'$} also holds if the irreducible factors of $(W,S)$ include 
%, beyond the cases listed in the theorem,  
  Coxeter systems of  types $A_7$ or $A_8$.  Properties \descref{C$'$} and \descref{D$'$}  hold if the irreducible factors of $(W,S)$ include  Coxeter systems of types $A_6$ or $I_2(m)$ for $m\leq 100$. 
  It is of course expected that Properties \descref{C$'$} and \descref{D$'$}  hold for all Coxeter systems of rank two, and this can probably be shown by  some technical but elementary calculations in the dihedral case. We do not attempt to carry these out in the present work, however. %We do not have the space to undertake this labor in the present work, however.
%\end{remark}

\subsection*{Acknowledgements}

I am grateful  to Ben Elias, George Lusztig, and David Vogan for many helpful discussions and suggestions.

\section{Computations for arbitrary Coxeter systems}\label{recover-sect}

Here we review some  general information about the distinguished bases of $\cH_q$ and $\cH_{q^2}$ and $\cM_{q^2}$.

\def\msig{m^\sigma}

\subsection{Lusztig's recurrence for the twisted polynomials}\label{recurrence-sect}

Let $(W,S)$ denote  a Coxeter system and  $* \in \Aut(W)$   any $S$-preserving involution of $W$. 
To describe how the standard basis $\(T_w\)_{w \in W}$ of $\cH_{q^2}$ acts on the distinguished basis $\(A_w\)_{w \in \I}$ of $\cM_{q^2}$, we introduce the following notation.

\begin{notation}
Recall that $q=v^2$ and $P^\sigma_{y,w} \in \ZZ[q]$.
Given $y, w \in \I$, let 
 $\mu^\sigma(y,w) 
 $ and $
 \nu^\sigma(y,w)$ respectively denote the coefficients 
of $v^{\ell(w)-\ell(y)-1}$ and $v^{\ell(w)-\ell(y)-2}$ in $P^\sigma_{y,w}$.
%
%\[
%\ba
% \mu^\sigma(y,w) &\omdef= \text{the coefficient of  $v^{\ell(w)-\ell(y)-1}$ in $P^\sigma_{y,w}$}
% \\
%  \nu^\sigma(y,w) &\omdef= \text{the coefficient of $v^{\ell(w)-\ell(y)-2}$ in $P^\sigma_{y,w}$.}
%  \ea
%\]
In turn, for each
 $s \in S$  define  $\mu^\sigma(y,w;s)$ as the integer given by
\[ \mu^\sigma(y,w;s) \omdef= 
\nu^\sigma(y,w)  + \delta_{sy,ys^*} \mu^\sigma(sy,w) -  \delta_{sw,ws^*} \mu^\sigma(y,sw) 
 - \sum_{{x \in \I;\hs sx<x }} \mu^\sigma(y,x) \mu^\sigma(x,w).
 \]
%Note that the sum on the right hand side is only over  $x \in \I^-(w,s)$ with $y<x<w$, since we have $\mu^\sigma(y,x) =0$ whenever $ y \not < x$.
As usual,  the Kronecker delta here means $\delta_{a,b} =1$ if $a=b$ and $\delta_{a,b} = 0$ otherwise. 
Note %, since $\Psig_{y,w}$ is a polynomial in $q=v^2$,
 that $\mu^\sigma(y,w)$ (respectively, $\nu^\sigma(y,w)$) is nonzero only if $y\leq w$ and $\ell(w) - \ell(y)$ is odd (respectively, even).
Define $\msig(y \xrightarrow{s}w) \in \cA$ for $y,w \in \I$ and $s \in S$ as the Laurent polynomial 
\be\label{msig-def}
\msig(y \xrightarrow{s}w) %= \mu^\sigma(y,w)(v+v^{-1}) + \mu^\sigma(y,w;s). 
= \begin{cases} \mu^\sigma(y,w)(v+v^{-1}) &\text{if $\ell(w) - \ell(y)$ is odd} \\ \mu^\sigma(y,w;s)&\text{if $\ell(w) - \ell(y)$ is even.}\end{cases}
\ee
Finally, let $\Des(w)  = \{ s \in S : \ell(sw) < \ell(w)\}$
and
$\mathrm{Des}_R(w)  = \{ s \in S : \ell(ws) < \ell(w)\}$.
\end{notation}%\be\label{des-def}\Des(w) \omdef = \{ s \in S : \ell(sw) < \ell(w)\}\qquand
%\mathrm{Des}_R(w) \omdef = \{ s \in S : \ell(ws) < \ell(w)\}\ee
Lusztig proves the following as \cite[Theorem 6.3]{LV2}.

\begin{theorem}[Lusztig \cite{LV2}] \label{mult-thm}
Let $w \in \I$ and $s \in S$. 
Then $C_s = q^{-1}(T_s + 1)$ and 
\[C_s A_w = \begin{cases} 
\(q+q^{-1}\) A_w&\text{if $s \in \Des(w)$} \\[-10pt]\\
 \(v+v^{-1}\) A_{s w} +\sum_{{y \in \I ;\hs  sy<y<s w}}  \msig(y \xrightarrow{s}w) A_y&\text{if $s \notin\Des(w)$ and $sw=ws^*$} \\[-10pt]\\
 A_{s w s^*} +\sum_{{y \in \I; \hs sy<y<s ws^*}}  \msig(y \xrightarrow{s}w) A_y&\text{if $s \notin \Des(w)$ and $sw\neq ws^*$}.\end{cases}\]
\end{theorem}

Comparing  coefficients of $a_y$ on both sides of the preceding equation yields our next result, which also appears as \cite[Corollary 2.7]{EM1}.
% Rewriting the right hand side in the standard basis $\(a_w\)_{ w\in \I}$ is straightforward from the definitions in Section \ref{twisted-intro}, while rewriting the left hand side can be done using the identities $C_s = q^{-1}(T_s +1)$ and $A_w = v^{-\ell(w)} \sum_{y \in \I} \Psig_{y,w} a_y$ with the multiplication rule \eqref{module-def}.
%In this manner one obtains the following corollary:

\begin{corollary}\label{main-recurrence}
Let $y,w \in \I$ with $y \leq w$  and  $s \in \Des(w)$.  

\begin{enumerate}

\item[(a)] $\Psig_{y,w} = \Psig_{s\act y,w}$.

\item[(b)] If $ s \in \Des(y)$ and  we let $w'=s\act w$ and  $c = \delta_{sw,ws^*}$ and $d = \delta_{sy,ys^*}$, then
{\small
\be\label{partb}  (q+1)^c \Psig_{y,w} = 
(q+1)^d P_{s\act y,w'}^\sigma  +  q(q-d)P_{ y,w'}^\sigma 
- 
\sum_{\substack{z \in \I;\hs sz<z \\ y \leq z < w}} v^{\ell( w)-\ell(z)+c}\cdot \msig(z \xrightarrow{s} w') \cdot P_{ y,z}^\sigma.
\ee
}
\end{enumerate}
\end{corollary}

%\begin{remark} Because of the $(q+1)^c$ factor on the left, it is not obvious from the recurrence in part (b) that $\Psig_{y,w} \in \cA$ is actually a polynomial in $q$ (but it is clear that $\Psig_{y,w}$ is a rational function of $q$). That $\Psig_{y,w} \in \ZZ[q]$ follows from results in \cite{LV2}, however, and in light of this, the recurrence shows that $(q+1)^{\ell^*(w)-\ell^*(y)}$ divides $\Psig_{y,w}$. Recall here the definition of  $\ell^*$ from Proposition-Definition \ref{ell-star}.
%\end{remark}

 The preceding theorem and corollary give recursive formulas which can be used to compute the polynomials $\Psig_{y,w}$ and $h^\sigma_{x,y;z}$, though there is some subtlety in how to go about this. 
The following algorithms carrying out such calculations are implemented in our extensions \cite{MyCode} to du Cloux's program {\tt Coxeter} \cite{Coxeter}. 

  \begin{algo-psig}% for computing the polynomials $\(\Psig_{y,w}\)_{y,w \in \I}$
If $y\not < w$ then $\Psig_{y,w} = \delta_{y,w}$. If $\ell(w) > 1$ then it is usually possible to compute $\Psig_{y,w}$ inductively 
by applying Corollary \ref{main-recurrence} in a straightforward way, since this result usually gives a formula for $\Psig_{y,w}$ in terms of polynomials $\Psig_{y',w'}$ with either $y' \leq w' < w$ or $y<y'<w'=w$, which
one may assume to be already known.
 However, 
 terms on the right hand side of the recurrence \eqref{partb}  can sometimes depend on terms on the left. There is actually only one such term: the summand  indexed by $z=y$ when $sw=ws^*$ and  $\ell(w)- \ell(y)$ is odd. In this case,  Corollary \ref{main-recurrence}(b) assumes the form
\be\label{thiseq}  (q+1) \Psig_{y,w} = f + v^{\ell(w)-\ell(y)+1} \mu^\sigma(y,w)\ee
where  $f \in \ZZ[q]$ is the polynomial
\[f =  (q+1)^d P_{s\act y,w'}^\sigma  +  q(q-d)P_{ y,w'}^\sigma 
- 
\sum_{\substack{z \in \I;\hs sz<z \\ y < z < w}} v^{\ell( w)-\ell(z)+c}\cdot \msig(z \xrightarrow{s} w') \cdot P_{ y,z}^\sigma\]
where our notation is as in Corollary \ref{main-recurrence}.
The definition of $f$ depends on quantities which one can assume to be already known,
and once $f$ is computed, it is straightforward to extract $\Psig_{y,w}$ from the equation \eqref{thiseq}. 
\end{algo-psig}

 \begin{algo-hsigma} By definition $h^\sigma_{1,y;z} = \delta_{y,z}$, and if $s \in S$ then $h^\sigma_{s,y;z}$ is determined by Theorem \ref{mult-thm}. When $x \in W$ has length greater than one so that there exists $s \in S$ with $sx<s$,  Theorem \ref{kl-thm} affords the recurrence
   \be\label{hsig-reccc}
  h^\sigma_{x,y;z}:= \sum_{z'\in \I} h^\sigma_{sx,y;z'}h^\sigma_{s,z';z}  - \sum_{\substack{x' \in W \\ sx'<x'<x}} \mu(x',sx) h^\sigma_{x',y;z}
  \ee
which expresses $h^\sigma_{x,y;z}$ in terms of quantities which may be assumed to have been already computed. Recall here that $\mu(x',sx)$ denotes the coefficient of $v^{\ell(x)-\ell(x')-2}$ in $P_{x',sx}$. 
%, which can be computed using {\tt Coxeter} \cite{Coxeter}.
   \end{algo-hsigma}

Du Cloux's papers \cite{Fokko0,Fokko} describe efficient algorithms (implemented in
  {\tt Coxeter} \cite{Coxeter}) for computing the Kazhdan-Lusztig polynomials $P_{y,w}$ and the structure constants $h_{x,y;z}$. 
Once the arrays $\(P_{y,w}\)_{y,w \in \I}$ and $\(\Psig_{y,w}\)_{y,w \in \I}$ have  been computed for a finite Coxeter system, it is straightforward to check Properties \descref{A$'$} and \descref{B$'$}.
To similarly check Properties \descref{C$'$} and \descref{D$'$}, one must first compute the arrays $\(h_{x,y;z}\)_{x,y,z\in W}$ and $\(h^\sigma_{x,y;z}\)_{x \in W,\hs y,z \in \I}$ and  then calculate the Laurent polynomials $\wt h_{x,y;z}$ via the identity
\be\label{tohtilde} \wt h_{x,y;z} = \sum_{z' \in W} h_{x,y;z'} h_{z',(x^*)^{-1},z}\qquad\text{for $x \in W$ and $y,z \in \I$.}\ee
Implementing this  formula  presents its own challenges in cases when the array $\(h_{x,y;z}\)_{x,y,z\in W}$ is very large; however, given $\wt h_{x,y;z}$ and $h^\sigma_{x,y;z}$ it is again straightforward to check our remaining positivity properties for a particular finite Coxeter system with involution. 
All of these checks are implemented in \cite{MyCode}.

% Before proceeding,
%  we mention two  additional properties of the polynomials $\Psig_{y,w}$, which follow from Corollary \ref{main-recurrence} in a straightforward manner by induction on $\ell(w)$.
%Lusztig states part (b)  explicitly as \cite[Proposition 4.10]{LV2}, but this actually serves in \cite{LV2} as a preliminary to the other results given here.
%
%
%\begin{corollary}[Lusztig \cite{LV2}]\label{lv-cor}
%Let $y,w \in \I$. % with $y\leq w$.
%\begin{enumerate}
%\item[(a)] $\Psig_{y,w} = \Psig_{y^{-1},w^{-1}}= \Psig_{y^*,w^*} = \Psig_{\tau(y),\tau(w)}$ for all $S$-preserving automorphisms $\tau \in \Aut(W)$ which commute with $*$ in the sense that $\tau(x^*) = \tau(x)^*$ for all $x \in W$.
%\item[(b)] $\Psig_{y,w}$ has constant coefficient 1 if $y\leq w$.
%\end{enumerate}
%\end{corollary}

\subsection{A special case of the twisted involution module}\label{specialcase-sect}

If $(W',S')$ and $(W'',S'')$ are Coxeter systems, then we define the direct product \[(W,S) = (W',S') \times (W'',S'')\] to be the Coxeter system with 
$ W = W'\times W'' $ and $ S = \{ (s,1) : s \in S'\} \cup \{ (1,s) : s \in S''\}.$
The following proposition shows that the Kazhdan-Lusztig polynomials $P_{y,w}$ can occur as instances of the twisted polynomials $\Psig_{y,w}$ for certain choices of $(W,S)$ and $*$. %This makes it possible to recover several important properties of the Kazhdan-Lusztig basis $\( c_w \)_{w \in W}\subset \cH_q$ and the Kazhdan-Lusztig polynomials $\(P_{y,w}\)_{w \in W}\subset \ZZ[q]$ from analogous statements for $\cM_{q^2}$.

\begin{proposition}\label{mod-prop} Suppose that  $(W,S) = (W',S') \times (W',S')$ 
 for some Coxeter system $(W',S')$,
and that  $* \in \Aut(W)$ acts by 
$(x,y)^* = (y,x)$ for $y,w \in W'$.
Let  $\cH'_{q^2}$ denote the  Hecke algebra of $(W',S')$ with parameter $q^2$, and define  \[\iota : \cH'_{q^2} \to \cH_{q^2} \qquand \iota^\sigma : \cH'_{q^2} \to \cM_{q^2}\]
as the unique $\cA$-linear maps with
$\iota(T_w )= T_{(w,1)}$ and $ \iota^\sigma(T_w ) =  a_{(w,w^{-1})}$ for $w \in W'$.

\ben

%\item[(a)] The map $w \mapsto (w,w^{-1})$ defines a poset isomorphism $(W',\leq) \xrightarrow{\sim} (\I,\leq)$.
\item[(a)] 
The map $\iota$ is injective and the map $\iota^\sigma$ is bijective, and for all $T,T ' \in \cH_{q^2}'$ we have
\[\iota(TT') = \iota(T)\iota(T') \qquand  \iota^\sigma(TT') = \iota(T) \iota^\sigma(T')\qquand \iota^\sigma(\hs \overline{T}\hs) =  \overline {\iota^\sigma(T)}.\]

\item[(b)] For all $y,w\in W'$, we have  $ \Psig_{(y,y^{-1}),(w,w^{-1})} = P_{y,w}(q^2).$

 \een

\end{proposition}

%\begin{remark} In part (c),  the   expression $P_{y,w}(q^2)$ denotes a Kazhdan-Lusztig polynomial  attached to the Coxeter system $(W',S')$ with its parameter $q$ replaced by $q^2$.
%%the left hand expressions are (Laurent) polynomials attached to the Coxeter system $(W,S)$, while the right expressions are quantities  defined in terms of the (Laurent) polynomials %$P_{y,w}$, $h_{x,y;z}$,  $\wt h_{x,y;z}$  attached to the Coxeter system $(W',S')$. In particular, $h_{x,y;z}(v^2)$ denotes the Laurent polynomial obtained by replacing the parameter $v$ with $q=v^2$ in $h_{x,y;z} \in \cA$.
%\end{remark}

\begin{proof}
The map $w \mapsto (w,w^{-1})$ clearly defines a poset isomorphism $(W',\leq) \xrightarrow{\sim} (\I,\leq)$, as is noted in \cite[Example 3.2]{H2} and also \cite[Example 10.1]{R}. 
It follows that $\iota^\sigma$ is a bijection. On the other hand, it is easy to check that the map $\iota$ is an injective $\cA$-algebra homomorphism.
Now, since $s\act w= sws^*$ for all $s \in S$ and $w \in W$,  it follows from Theorem-Definition \ref{lv-thmdef1}  that 
$ \iota^\sigma(T_s T_w) = \iota(T_s) \iota^\sigma(T_w) = T_{(s,1)} a_{(w,w^{-1})}$ for all $s \in S'$ and $w \in W'$.
Since $\iota$ is a homomorphism and $\iota^\sigma$ is $\cA$-linear, we conclude that $\iota^\sigma(TT') = \iota(T) \iota^\sigma(T')$ for all $T,T' \in \cH_q'$.
One can similarly check that $\iota^\sigma(\overline {T_w}) = \overline{\iota^\sigma(T_w)}$ for $w \in W'$, which establishes the last assertion in part (a) by $\cA$-linearity.
%In turn, we note that  if $w \in W'$ then $T_w = T_w \cdot T_1$ and $a_{(w,w^{-1})} = T_{(w,1)}\cdot a_1$, so 
%\[ \iota^\sigma(\overline {T_w})= \iota^\sigma( T_{w^{-1}}^{-1} \cdot T_1) = \iota(T_{w^{-1}}^{-1}) \iota^\sigma(T_1) = T_{(w,1)^{-1}}^{-1} \cdot a_1 = \overline{T_{(w,1)}} \cdot \overline{a_1} = \overline {a_{(w,w^{-1})}} = \overline{\iota^\sigma(T_w)}.\] Part (a) now follows by $\cA$-linearity.

To prove part (b), 
note that  $\ell((w,w^{-1})) = 2\ell(w)$ for $w \in W'$, where on the left $\ell$ is interpreted as the length function of $(W,S)$ and on the right as the length function of $(W',S')$. 
%It follows from this and 
% Theorem-Definition \ref{kl-thmdef} that  $\(P_{y,w}(q^2)\)_{y,w \in W'}$ is a family of polynomials with the following properties: 
%\begin{itemize}
%\item For each $w \in W'$ the sum $v^{-\ell((w,w^{-1}))} \cdot \sum_{y \in W'} P_{y,w}(q^2) \cdot a_{(y,y^{-1})} \in \cM_{q^2}$ is equal to $\iota^\sigma(C_w)$ and so is invariant under the bar operator;
%\item $P_{y,w}(q^2) = \delta_{(y,y^{-1}),(w,w^{-1})}$ if $y \not < w$;
%\item $P_{y,w}(q^2)$ has degree at most $\frac{1}{2} (\ell((w,w^{-1})) - \ell((y,y^{-1}))-1)$ as a polynomial in $q$ if $y<w$.
%\end{itemize}
In light of this and Theorem-Definition \ref{kl-thmdef}, 
it follows from the uniqueness specified in Theorem-Definition \ref{lv-thmdef3}  
that %$\iota^\sigma(C_w) = A_{(w,w^{-1})}$ whence 
$ \Psig_{(y,y^{-1}),(w,w^{-1})} = P_{y,w}(q^2)$ for $y,w \in W'$.
\end{proof}

When $(W,S,*)$ is as in Proposition \ref{mod-prop}, one can express the polynomials $P^\pm_{y,w}$ and $h^\pm_{x,y;z}$ entirely in terms of polynomials attached to $(W',S')$. 
To prove these formulas, we require two lemmas.
Our first lemma 
applies to an arbitrary Coxeter system $(W,S)$ with an $S$-preserving involution $* \in \Aut(W)$.

\begin{lemma}\label{hstar-lem} If $x\in W$ and $y,z \in \I$ then $h^\sigma_{x,y;z} = h^\sigma_{x^*,y^*;z^*}$.
\end{lemma}

\begin{proof}
Let $\varphi : \cH_{q^2} \to \cH_{q^2}$ and $\varphi^\sigma : \cM_{q^2} \to \cM_{q^2}$ be the unique $\cA$-linear maps with $\varphi(T_w) = T_{w^*}$ and $\varphi^\sigma(a_w) = a_{w^*}$. It is clear that $\varphi(T_s) \varphi^\sigma(a_w) = \varphi^\sigma(T_s a_w)$ for all $s \in S$ and $w \in \I$.
Because $*$ is an automorphism of $W$ preserving $S$, the map $\varphi$ is an automorphism of the $\cA$-algebra $\cH_{q^2}$, and from this  and $\cA$-linearity it follows  
%that $\varphi(T_x) \varphi^\sigma(a_w) = \varphi^\sigma(T_x a_w)$ for all $x \in W$ and $w \in \I$, and by $\cA$-linearity we conclude
 that $\varphi(T) \varphi^\sigma(a) = \varphi^\sigma(Ta)$ for all $T \in \cH_{q^2}$ and $a \in \cM_{q^2}$.
It is a straightforward exercise to show that  
\[\varphi^\sigma(A_y) = A_{y^*}\text{ for all $y \in \I$} \qquand  \varphi(C_x) = C_{x^*}\text{ for all $x \in W$}\] and consequently 
$  \varphi^\sigma\(\sum_{z \in W} h^\sigma_{x,y;z} A_z\) = \varphi^\sigma(C_x A_y) =\varphi(C_x) \varphi^\sigma( A_y) = C_{x^*} A_{y^*}.$
The left hand side of the preceding equation is  equal to $\sum_{z \in W} h^\sigma_{x,y;z} A_{z^*} $ while the right hand side is equal to $\sum_{z \in W} h^\sigma_{x^*,y^*;z^*} A_{z^*}$, so since $\(A_z\)_{z \in \I}$ is a basis for $\cM_{q^2}$ we conclude that $h^\sigma_{x,y;z} = h^\sigma_{x^*,y^*;z^*}$.
\end{proof}

Our second lemma applies only to the special situation of Proposition \ref{mod-prop}.

\begin{lemma}\label{lastbit-lem}
Suppose that $(W,S)$, $(W',S')$, and $*$ are defined as in Proposition \ref{mod-prop}. 
\ben

\item[(a)] For all $y,y',w,w' \in W'$, we have $P_{(y,y'),(w,w')} = P_{y,w} P_{y',w'}$, and 
   it holds  that 
\[c_{(w,w')} = c_{(w,1)} c_{(1,w')} = c_{(1,w')} c_{(w,1)} 
\quand  
C_{(w,w')} = C_{(w,1)} C_{(1,w')} = C_{(1,w')} C_{(w,1)}
\]
 in the respective Hecke algebras $\cH_{q}$ and $\cH_{q^2}$ attached to $(W,S)$.

\item[(b)] For all $x,y,z \in W'$, we have $h_{(x,1),(y,1);(z,1)} = h_{(1,y^{-1}),(1,x^{-1});(1,z^{-1})} = h_{x,y;z}$.

\item[(c)] For all $x,y,z\in W'$, we have  $h^\sigma_{(x,1),(y,y^{-1});(z,z^{-1})} = h^\sigma_{(1,y^{-1}),(x,x^{-1});(z,z^{-1})} = h_{x,y;z}(v^2).$

\een
\end{lemma}

\begin{proof}
Parts (a) and (b) follow as a straightforward exercise from Theorem-Definition \ref{kl-thmdef}.
%, using (1) the fact that $t_w \mapsto t_{(w,1)}$ and $t_w \mapsto t_{(1,w)}$ are $\cA$-algebra embeddings of the Hecke algebra of $(W',S')$ in the Hecke algebra of $(W,S)$, which commute with all relevant bar operators, and (2) the fact that $t_{(w,w')} = t_{(w,1)} t_{(1,w')} = t_{(1,w')} t_{(w,1)}$ in $\cH_{q}$ for all $w,w' \in W'$, with a similar formula holding in $\cH_{q^2}$.
%In light of Corollary \ref{kl-cor}(c), part (b) follows  from the fact implied by part (a) that the $\cA$-algebra embeddings with $t_w \mapsto t_{(w,1)}$ and $t_w \mapsto t_{(1,w)}$ for $w \in W'$ send $c_w$  to $c_{(w,1)}$ and $c_{(1,w)}$ respectively. 
%
%
To prove part (c),  recall the definitions of the maps $\iota$ and $\iota^\sigma$ above, and note from part (a) and Proposition \ref{mod-prop}(c) that for  $w \in W'$ we have $\iota(C_w) = C_{(w,1)}$ and $\iota^\sigma(C_w) = A_{(w,w^{-1})}$.
Therefore
 for $x,y,z \in W'$ it holds that 
\[ \sum_{z \in W'} h^\sigma_{(x,1),(y,y^{-1});(z,z^{-1})} A_{(z,z^{-1})} = C_{(x,1)} A_{(y,y^{-1})} =\iota^\sigma(C_xC_y) =   \sum_{z \in W'} h_{x,y;z}(v^2) A_{(z,z^{-1})}. \]
Note that on the right  $h_{x,y;z}$ is evaluated at $v^2$ rather than at $v$ because we are computing the product $C_xC_y$ in a Hecke algebra with parameter $q^2$ rather than $q$.
Thus $h^\sigma_{(x,1),(y,y^{-1});(z,z^{-1})} =  h_{x,y;z}(v^2)$.
 Combining this fact with the well-known identity $h_{x,y;z} = h_{y^{-1},x^{-1};z^{-1}}$ (see \cite[Section 2.1]{Fokko}) and Lemma \ref{hstar-lem} gives the 
 second equality in part (c).
% finally gives 
%\[ h^\sigma_{(1,y^{-1}),(x,x^{-1});(z,z^{-1})} = h^\sigma_{(y^{-1},1),(x^{-1},x);(z^{-1},z)} = h_{y^{-1},x^{-1};z^{-1}}(v^2) = h_{x,y;z}(v^2).\]
\end{proof}

We may now state the main result of this section.

\begin{proposition}\label{recovery-thm} Suppose that $(W,S)$, $(W',S')$, and $*$ are defined as in Proposition \ref{mod-prop}. 

\ben
\item[(a)]  The set  of polynomials $\left\{ P^\pm_{y,w} : y,w \in \I\right\}$ is equal to
\[ %\Bigl\{ P^\pm_{y,w} : y,w \in \I\Bigr\} =
 \Bigl\{ \tfrac{1}{2}\( P_{y,w}(q)^2 \pm P_{y,w}(q^2) \) : y,w \in W' \Bigr\}.\]

\item[(b)]  The set of polynomials $\left\{ P^\pm_{y,w} - P^\pm_{z,w} : y,z,w \in \I,\ y \leq z\right\}$ is equal to
\[ %\Bigl\{ P^\pm_{y,w} - P^\pm_{z,w} : y,z,w \in \I,\ y \leq z\Bigr\} =
 \Bigl\{ \tfrac{1}{2}\( P_{y,w}(q)^2-P_{z,w}(q)^2\) \pm \tfrac{1}{2}\(P_{y,w}(q^2) - P_{z,w}(q^2) \) : y,z,w \in W',\ y\leq z \Bigr\}.
 \]

\item[(c)] The set  of polynomials $\left\{ h^\pm_{x,y;z} : x \in W,\ y,z \in \I \right\}$ is equal to
\[ %\Bigl\{ h^\pm_{x,y;z} : x \in W,\ y,z \in \I \Bigr\} =
\Bigl\{ \tfrac{1}{2}\( f_{w,x,y;z}(v)^2 \pm f_{w,x,y;z}(v^2)  \) : w,x,y,z\in W' \Bigr\}
\]
where we define $f_{w,x,y;z} = \sum_{g\in W'} h_{w,x;g} h_{g,y;z} \in \cA$ for $w,x,y,z \in W'$.

\een

\end{proposition}

\begin{remark}
Observe that the polynomials $f_{w,x,y;z}$  defined in part (c) of this result are  the structure constants satisfying $c_w c_x c_y = \sum_{z \in W'} f_{w,x,y;z} c_z$ for $w,x,y \in W'$.
\end{remark}

\begin{proof}
Parts (a) and (b) follow by comparing the definition of the polynomials $P^\pm_{y,w}$ for $y,w \in \I$ with Proposition \ref{mod-prop}(c) and Lemma \ref{lastbit-lem}(a), while noting the well-known identity $P_{y,w} = P_{y^{-1},w^{-1}}$ \cite{KL}.
To prove part (c) it suffices to show that for $w,x,y,z \in W'$ we have 
\be\label{f-eq} \wt h_{(w,y^{-1}),(x,x^{-1});(z,z^{-1})} = f_{w,x,y;z}(v)^2 \qquand h^\sigma_{(w,y^{-1}),(x,x^{-1});(z,z^{-1})} = f_{w,x,y;z}(v^2).\ee
To check the left identity, we compute from Lemma \ref{lastbit-lem}(a) that
\[  c_{(w,y^{-1})} c_{(x,x^{-1})} c_{(y,w^{-1})} =\( c_{(w,1)} c_{(x,1)} c_{(y,1)}\) \cdot \(c_{(1,y^{-1})} c_{(1,x^{-1})} c_{(1,w^{-1})} \).
\]
Applying Lemma \ref{lastbit-lem}(b) to the products $ \(c_{(w,1)} c_{(x,1)}\) \cdot c_{(y,1)} $ and $c_{(1,y^{-1})} \cdot \(c_{(1,x^{-1})} c_{(1,w^{-1})}\) $, noting our parenthesizations, shows that
\[ 
 c_{(w,1)} c_{(x,1)} c_{(y,1)} =\sum_{z \in W'} f_{w,x,y;z} c_{(z,1)} 
\qquand 
c_{(1,y^{-1})} c_{(1,x^{-1})} c_{(1,w^{-1})} = \sum_{z \in W'} f_{w,x,y;z} c_{(1,z^{-1})}.
\]
By Lemma \ref{lastbit-lem}(a) we thus have $c_{(w,y^{-1})} c_{(x,x^{-1})} c_{(x,w^{-1})}  = \sum_{z,z' \in W'} f_{w,x,y;z}  f_{w,x,y;z'} \cdot c_{(z',z^{-1})}$, from which the first  identity in \eqref{f-eq} follows.
To check the second equality in \eqref{f-eq}, we note from Lemma \ref{lastbit-lem}(a) that
$C_{(w,y^{-1})} A_{(x,x^{-1})} = C_{(1,y^{-1})}C_{(w,1)}  A_{(x,x^{-1})}$, which implies  by Lemma \ref{lastbit-lem}(c) that
\[ h^\sigma_{(w,y^{-1}),(x,x^{-1});(z,z^{-1})} =  \sum_{g \in W'} h^\sigma_{(w,1),(x,x^{-1});(g,g^{-1})}  h^\sigma_{(1,y^{-1}),(g,g^{-1});(z,z^{-1})} = f_{x,w,y;z}(v^2)\]
%\[ \ba C_{(w,y^{-1})} A_{(x,x^{-1})}& = C_{(1,y^{-1})}C_{(w,1)}  A_{(x,x^{-1})}
%\\
%&
% = \sum_{g \in W'} h^\sigma_{(w,1),(x,x^{-1});(g,g^{-1})} C_{(1,y^{-1})}  A_{(g,g^{-1})}
%\\
%&
%= \sum_{z \in W'} \sum_{g \in W'} h^\sigma_{(w,1),(x,x^{-1});(g,g^{-1})}  h^\sigma_{(1,y^{-1}),(g,g^{-1});(z,z^{-1})}   A_{(z,z^{-1})}
%.\ea
%\]
as desired. Therefore both identities in \eqref{f-eq} hold so part (c) holds.
\end{proof}

\section{Reduction to the irreducible case}\label{reduction-sect}

We devote this section to the proof of Theorem \ref{x-thm} from the introduction. Our proof 
depends on a few preliminary facts, which occupy the next three subsections.

\subsection{Facts about unimodal polynomials}\label{unimodal-sect}

\def\darga{\mathrm{darga}}

Recall the definition of unimodality from Section \ref{conj-sect}. In particular, note that if a polynomial $f \in \ZZ[x]$ is  unimodal then automatically $f \in \NN[x]$.
Let $f \in \cA$ be a Laurent polynomial with degree $d$ in $v$. We say that 
$f $ is \emph{balanced} 
if  $v^d f$ is polynomial in $q=v^2$ such that
\[ v^d f = a_0 + a_1 q + a_2q^2 + \dots + a_d q^d\]
for some integers $a_i \in \ZZ$ with $a_i = a_{d-i}$ for all $0\leq i \leq d$.
We say that $f $ is \emph{balanced unimodal} if $f$ is a balanced and additionally $v^df$ is a unimodal polynomial in $q$. 
Note that 0 is balanced unimodal since we consider the zero polynomial to have degree 0.

\begin{lemma}\label{bal-lem}
Suppose $f,g \in \cA$ are nonzero and balanced unimodal.
Then the product $fg$ is balanced unimodal,
while the sum $f+g$ is balanced unimodal if and only if 
the degrees of $f$ and $g$ as polynomials in $v$ are either both even or both odd.
%
%\ben
%\item[(a)] The product $fg$ is balanced unimodal.
%\item[(b)] Assume $f $ and $g$ are both nonzero; then the sum $f+g$ is balanced unimodal if and only if the degrees of $f$ and $g$ as polynomials in $v$ are either both even or both odd.
%
%\een
\end{lemma}

\begin{proof}
Suppose $f$ and $g$ have degrees $d$ and $d'$ as Laurent polynomials in $v$. 
Then $v^d f$ and $v^{d'} g$ are ``symmetric unimodal'' in the sense of \cite{zel} and it follows that $fg$ is balanced unimodal by \cite[Observation 2]{zel}.
The remainder of the lemma follows by inspection.
%
%The desired properties follow by applying Lemma \ref{zel-lem}  to $v^df$ and $v^{d'}g$.
%%
%%%
%%To prove part (a), note that the product $fg$ then has degree $d+d'$, and that $v^{d+d'} fg = (v^d f) (v^{d'} g)$ is a product of symmetric unimodal polynomials in $q$, so by Lemma \ref{zel-lem}(a), $v^{d+d'} fg $ is  a symmetric unimodal polynomial in $q$. Thus $fg$ is balanced unimodal.
%%
%%To prove part (b), without loss of generality assume $d' \leq d$, so that $f+g$ has degree $d$  in the indeterminate $v$. If $d$ and $d'$ have distinct parities then $v^d(f+g)$ is not a polynomial in $q$ so is not balanced unimodal. If $d$ and $d'$ have the same parity then $v^dg$ and $v^df$ are both polynomials in $q$ which are symmetric unimodal with  darga $d$, in which case  is  a symmetric unimodal polynomial in $q$ by Lemma \ref{zel-lem}(b), $v^d(f+g)$, so $f+g$ is balanced unimodal.
\end{proof}

We also require the following technical  lemma.
 
 \begin{lemma}\label{myown-lem} Let $f \in \cA$ and define $f^\pm =  \frac{1}{2}\(f(v)^2 \pm f(v^2)\)$.
 \ben
 \item[(a)] 
Both $f^+$ and $f^-$   belong to $\cA$.

\item[(b)] If $f$ has nonnegative coefficients then  $f^+$ and $f^-$ have nonnegative coefficients.

\item[(c)] If $f$ is balanced unimodal then  $f^+$ and $f^-$ are balanced unimodal.
\een
\end{lemma}

\begin{proof}
%If we write $f = \sum_{i \in \ZZ} c_i v^i$ where each $c_i \in \ZZ$ %(and only finitely many $c_i$'s are nonzero),
% then 
%$f^\pm = \sum_{i \in \ZZ} \frac{1}{2}(c_i^2 \pm c_i) v^{2i} +  \sum_{ i<j} (c_ic_j)v^{i+j}.$
%For all $c \in \ZZ$ it holds that  $c^2 \pm c$ is an even nonnegative integer, so $f^\pm \in \cA$ and if $f \in \NN[v,v^{-1}]$ then  $f^\pm \in \NN[v,v^{-1}]$ as well. This proves 

Parts (a) and (b) follow by computing the coefficients of $f^\pm$ in terms of those of $f$.
To prove part (c), suppose $f \in \cA\setminus\{0\}$ is balanced unimodal. We may assume $f$ is nonzero with degree $d$ as a polynomial in $v$. We need only show that $ 2f^+$ and $2f^-$ are balanced unimodal. To this end, note that $f$ is  a  linear combination with nonnegative integer coefficients of polynomials of the form $h_i \omdef = v^{-d}(q^i + \dots + q^{d-i})$ for integers $0 \leq i \leq d/2$. In particular, we may write $f = \sum_{0 \leq i \leq d/2} a_i h_i$ for some nonnegative integers $a_i \in \NN$ with $a_0 \neq 0$, and we then have
\be\label{temp-eq} 2f^\pm= \sum_{0 \leq i \leq d/2}  a_i(h_i^2 \pm  h_i(v^2)) +\sum_{0\leq i \leq d/2} (a_i^2-a_i) h_i^2+  \sum_{ 0\leq  i<j \leq d/2} (2a_ia_j)h_ih_j.\ee
To show that $2f^\pm$ is balanced unimodal, it is enough by Lemma \ref{bal-lem} to check that the terms $h_i^2 \pm  h_i(v^2)$ and $h_i^2$ and $h_ih_j$ occurring in the three sums in \eqref{temp-eq} are  balanced unimodal polynomials whose degrees in $v$ are all even.
This is a simple exercise, which we leave to the reader.
%%(Note that the  coefficients $a_i$ and $a_i^2-a_i$ and $2a_ia_j$ corresponding to these terms are all nonnegative integers.)
% Lemma \ref{bal-lem}(a) implies directly that the 
%products $h_i^2$ and $h_ih_j$ are balanced unimodal; moreover, these polynomials both have even degrees in $v$, given by $2d-4i$ and $2d-2i-2j$. In turn, it is straightforward to compute that 
% \[ h_i^2 =  \sum_{k=2i-d}^{d-2i}  b_k q^{k}\qquand  h_i(v^2) =  \sum_{k=2i-d}^{d-2i}  c_k q^{k}\]
% where $b_{2i-d},b_{2i-d+1},\dots,b_{d-2i}$ and $c_{2i-d},c_{2i-d+1},\dots,c_{d-2i}$ are respectively the sequences of integers
%\[ 1,2,3,\dots,d-i,d-i+1,d-i,\dots,3,2,1\qquand 1,0,1,0,\dots,0,1,0,1.\]
%It follows by inspection that  $h_i^2 + h_i(v^2)$ and $h_i^2 - h_i(v^2)$ are both balanced unimodal polynomials with even degrees in $v$. By Lemma \ref{bal-lem}(b) applied to \eqref{temp-eq}, we conclude    that the polynomials $2f^\pm$ are both balanced unimodal, as desired.
\end{proof}

\subsection{Facts about the structure constants}

In both propositions in this section, we let $u = v+v^{-1}$ and we let $(W,S)$ denote an arbitrary Coxeter system with an $S$-preserving involution $*$.

  \begin{proposition}\label{h-vv}
%Let 
%$(W,S)$ be any Coxeter system, set
% $u=v+v^{-1}$, and 
 Suppose  $x,y,z \in W$.
\ben
\item[(a)] If $\ell(x)+\ell(y) + \ell(z) $ is odd then $h_{x,y;z} \in u\ZZ[u^2]$.
\item[(b)] If $\ell(x)+\ell(y)+\ell(z)$ is even then $h_{x,y;z} \in \ZZ[u^2]$.
\een
%In particular, if $d$ is the degree of $h_{x,y;z}$ as a polynomial in $v$, then $v^d h_{x,y;z} \in \ZZ[q]$ is a palindromic polynomial in $q$.
\end{proposition}

%\begin{remark}
%Observe that in case (a), we have for some odd $d \in \NN$ and some integers $a_i \in \ZZ$ that \[h_{x,y;z} = a_1(v+v^{-1}) + a_3(v^3+v^{-3}) + \dots + a_d (v^d+v^{-d}),\]  while in case (b) we have for some even $d \in \NN$ and some integers $a_i \in \ZZ$ that
%\[ h_{x,y;z} = a_0 + a_2(v^2+v^{-2}) + \dots + a_d (v^d + v^{-d}).\]  In either case, it holds that $v^d h_{x.y;z} \in \ZZ[q]$ is a  polynomial in $q$, which is \emph{palindromic} in the sense that its sequence of coefficients is the same read forwards and backwards.
%\end{remark}

\begin{proof}
Since $h_{1,y;z} = \delta_{y,z}$, the proposition holds when $\ell(x) = 0$, and since $\mu(z,y)$ is nonzero only if $\ell(y)-\ell(z)$ is odd,  Theorem \ref{kl-thm} shows that the proposition also holds when $\ell(x)=1$.
Assume $\ell(x) \geq 2$ and that the proposition holds if we replace $x$ by any element of shorter length. Choose $s \in \Des(x)$. 
By Theorem \ref{kl-thm} we have $c_x = c_s c_{sx} - \sum_{x' \in W;\hs sx'<x'<sx} \mu(x',sx)c_{x'}$, and so 
 \be\label{h-rec}h_{x,y;z}= \sum_{z'\in W} h_{sx,y;z'}h_{s,z';z}  - \sum_{\substack{x' \in W \\ sx'<x'<x}} \mu(x',sx) h_{x',y;z}.\ee
Since $\mu(x',sx)$ is nonzero only if $\ell(x) - \ell(x')$ is even, it follows  by  our inductive hypothesis that $\sum_{{x' \in W; \hs sx'<x'<sx}} \mu(x',sx) h_{x',y;z}$ belongs to $u\ZZ[u^2]$ or $\ZZ[u^2]$ if $\ell(x)+\ell(y) + \ell(z)$ is odd or even respectively.
On the other hand,   for all $z' \in W$ the parities of 
\[\ell(sx)+\ell(y) + \ell(z') %= \ell(x) + \ell(y) + \ell(z')-1
\qquand 
\ell(s)+\ell(z') + \ell(z)% = \ell(z) + \ell(z') +1
\]
are distinct or equal according to whether $\ell(x)+\ell(y)+\ell(z)$ is odd or  even respectively. Therefore it follows likewise by  hypothesis that  $\sum_{z'\in W} h_{sx,y;z'}h_{s,z';z}$ belongs to $u\ZZ[u^2]$ or $\ZZ[u^2]$ if $\ell(x)+\ell(y) + \ell(z)$ is odd or even respectively.
The proposition thus holds for all $x$ by \eqref{h-rec} and induction.
\end{proof}

%As a corollary to the preceding result, we have this next proposition.

 \begin{proposition}\label{rest-vv}
%The Laurent polynomials $\wt h_{x,y;z}$ belong to the subring $\ZZ[v+v^{-1}] \subset \cA$.
%Let $(W,S)$ be a Coxeter system with an $S$-preserving involution $* \in \Aut(W)$, set $u=v+v^{-1}$, and 
Suppose  $x \in W$ and $y,z \in \I$.
\ben
\item[(a)] If $\ell(y) + \ell(z) $ is even then $\wt h_{x,y;z} $ and $h^\sigma_{x,y;z} $ and $h^\pm_{x,y;z} $ all belong to $\ZZ[u^2]$.
\item[(b)] If $\ell(y)+\ell(z)$ is odd then $\wt h_{x,y;z} $ and $h^\sigma_{x,y;z} $ and $h^\pm_{x,y;z} $ all belong to $u\ZZ[u^2]$.
\een
%In particular, if $d$ is the degree of $\wt h_{x,y;z}$ as a polynomial in $v$, then $v^d \wt h_{x,y;z} \in \ZZ[q]$ is a palindromic polynomial in $q$.
\end{proposition}

%\begin{remark}
%The comments after Proposition \ref{h-vv} apply likewise to the polynomials $\wt h_{x,y;z}$; in particular, it follows that  if $d$ is the degree of $\wt h_{x,y;z}$ as a polynomial in $v$, then $v^d \wt h_{x,y;z} \in \ZZ[q]$ is a palindromic polynomial in $q$.
%\end{remark}

\begin{proof}
Since $\ell((x^*)^{-1}) = \ell(x)$, 
 the parities of $\ell(x)+\ell(y)+\ell(z')$ and $\ell(z') + \ell((x^*)^{-1})+\ell(z)$ are either always  equal or always distinct  for  $z' \in W$, according to whether $\ell(y) + \ell(z)$ is even or odd respectively. Since $\wt h_{x,y;z} = \sum_{z' \in W} h_{x,y;z'}h_{z',(x^*)^{-1};z}$, it follows from Proposition \ref{h-vv}  that $\wt h_{x,y;z}$ belongs to $\ZZ[u^2]$ if $\ell(y)+\ell(z)$ is even and to $u\ZZ[u^2]$ otherwise.
 
We next establish the claim that $h^\sigma_{x,y;z}$ belongs to $\ZZ[u^2]$ or $u\ZZ[u^2]$ according to whether $\ell(y) + \ell(z)$ is even or odd. The proof of this fact is similar to that of Proposition \ref{h-vv}.
Since $h^\sigma_{1,y;z} =\delta_{y,z}$ our claim holds if $\ell(x) = 0$. Since $\msig(z\xrightarrow{s}y)$ belongs to $\ZZ[u^2]$ if $\ell(y) + \ell(z)$ is even and to $u\ZZ[u^2]$ otherwise (see   \eqref{msig-def}), Theorem \ref{mult-thm} shows that our claim also holds when $\ell(x) \leq 1$. 
Finally, when $\ell(x) \geq 2$ and $s \in \Des(x)$,
our claim follows by induction using \eqref{hsig-reccc} exactly as in the proof of Proposition \ref{h-vv}.
% we have $C_x = C_s C_{sx} - \sum_{x' \in W;\hs sx'<x'<sx} \mu(x',sx)C_{x'}$ by  Theorem \ref{kl-thm}  so
%  \be\label{hsig-rec}h^\sigma_{x,y;z}= \sum_{z'\in \I} h^\sigma_{sx,y;z'}h^\sigma_{s,z';z}  - \sum_{\substack{x' \in W \\ sx'<x'<x}} \mu(x',sx) h^\sigma_{x',y;z}.\ee
%In this case our claim follows by induction exactly as in the proof of Proposition \ref{h-vv}.

Combining the preceding paragraphs demonstrates that 
the polynomials $h^\pm_{x,y;z}$, which automatically belong to $\cA = \ZZ[v,v^{-1}]$ by \cite[Proposition 2.11]{EM1}, also belong to   $\QQ[u^2] $ or $u\QQ[u^2]$ according to whether $\ell(y) + \ell(z)$ is even or odd. It is straightforward to check that 
$\cA \cap \QQ[u^2] \subset \ZZ[u^2]$ and $\cA \cap u\QQ[u^2] \subset u \ZZ[u^2]$, which establishes the proposition in full.
 \end{proof}

All elements of $\ZZ[u^2]$ have the form $a_0 + a_2(v^2+v^{-2}) + \dots + a_d (v^d + v^{-d})$  while all elements of $u\ZZ[u^2]$ have the form $a_1(v+v^{-1}) + a_3(v^3+v^{-3}) + \dots + a_d (v^d+v^{-d})$ for some integers $a_i \in \ZZ$. From this observation and the preceding propositions derives the following corollary.

\begin{corollary}\label{vv-cor} The Laurent polynomials $h_{x,y;z}$, $\wt h_{x,y;z}$, $h^\sigma_{x,y;z}$, $h^\pm_{x,y;z} \in \cA$ are always balanced. % in the sense of Section \ref{unimodal-sect}.
%; i.e., they become symmetric polynomials in $q$ when multiplied by $v$ to the power of their degrees.
\end{corollary}

\subsection{Reductions}

Propositions \ref{reduction-prop} and \ref{reductionprime-prop} in this section together imply  Theorem \ref{x-thm} in the introduction.
Before proceeding to these results we require two additional lemmas.

\begin{lemma}\label{recovery-thm2}
Suppose that $(W,S)$, $(W',S')$, and $*$ are defined as in Proposition \ref{mod-prop},
and let X be one of the letters A, B, C, or D. 
 If Property X holds for $(W',S')$, then Property X$'$ holds for the triple $(W,S,*)$.
% The following   implications then hold:
%\ben
%\item[(i)] If Property X holds for $(W',S')$, then Property X$'$ holds for the triple $(W,S,*)$.
%
%\item[(ii)] Assume X is B or D. If Property X$'$ holds for the triple $(W,S,*)$, then Property X$''$ holds for $(W',S')$.
%
%\een
\end{lemma}

\begin{proof}
%We first prove part (i). 
We know that Property \descref{A} holds for $(W',S')$, and 
it follows that Property \descref{A$'$}  holds for $(W,S,*)$ from Proposition \ref{recovery-thm}(a) and  Lemma \ref{myown-lem}(b).
Suppose Property \descref{B} holds for $(W',S')$. Fix  $y,z,w \in W'$ with $y \leq z$ and let $f=P_{y,w}$ and $g=P_{z,w}$. Then $f-g \in \NN[q]$ and also $f,g \in \NN[q]$, since Property \descref{B} implies Property \descref{A}, and so 
\[ (f(q)^2 -g(q)^2) \pm (f(q^2) - g(q^2)) = \underbrace{(f-g)^2 \pm (f(q^2)- g(q^2))}_{\in \NN[q]\text{ by Lemma \ref{myown-lem}(b)}} + \underbrace{2g(f-g)}_{ \in \NN[q]} \in \NN[q]. \]
Property \descref{B$'$} therefore holds for $(W,S,*)$ by Proposition \ref{recovery-thm}(b).

For the remainder of the proof,  fix arbitrary elements $w,x,y,z \in W'$ and write $f=f_{w,x,y;z}$ as in Proposition \ref{recovery-thm}(c). Then %$f_{w,x,y;z} = \sum_{g\in W'} h_{w,x;g} h_{g,y;z}$ for $w,x,y,z \in W'$ and that 
Properties \descref{C$'$} and \descref{D$'$} are  respectively equivalent to the assertions that the polynomials $f^\pm \omdef= \tfrac{1}{2} \( f(v)^2 \pm f(v^2)\)$
%\[ \tfrac{1}{2}\( f_{w,x,y;z}(v)^2 \pm f_{w,x,y;z}(v^2)\)\]
always have nonnegative coefficients and always are balanced unimodal.
%
%Let $w,x,y,z \in W'$ be arbitrary elements and write $f=f_{w,x,y;z}$. 
Since Property \descref{C} always holds  for $(W',S')$ we have $f \in \NN[v,v^{-1}]$ so   $f^\pm \in \NN[v,v^{-1}]$ by Lemma \ref{myown-lem}(b). 

Suppose     Property \descref{D} holds for $(W',S')$. The structure constants $h_{x,y;z}$ are then always balanced unimodal, and so by Lemma \ref{bal-lem}(a)  the product $h_{w,x;g} h_{g,y;z}$  for  each $g \in W'$ is likewise balanced unimodal.  Let $u = v+v^{-1}$. For all $w \in W'$, it holds by Proposition \ref{h-vv} that $h_{w,x;g} h_{g,y;z}$ belongs to $u\ZZ[u^2]$ or $\ZZ[u^2]$ according to whether  $\ell(w) + \ell(x)$ and $\ell(y) + \ell(z)$ have distinct or equal parities. Thus the degrees of the products $h_{w,x;g} h_{g,y;z}$ for $g \in W'$ all have the same parity, so $f$, being equal to sum of such products, is balanced unimodal by Lemma \ref{bal-lem}(b).
By Lemma \ref{myown-lem}(c) it follows that the polynomials  $f^\pm$ are therefore balanced unimodal, so Property \descref{D$'$} holds for $(W,S,*)$.
%
%The proof of part (ii) is simpler. For all $y,w \in W'$ we have
%$P^+_{(y,y^{-1}),(w,w^{-1})} + P^-_{(y,y^{-1}),(w,w^{-1})} %= P_{(y,y^{-1}),(w,w^{-1})} = P_{y,w} P_{y^{-1},w^{-1}} 
%=
% (P_{y,w})^2$, by the combination of Corollary \ref{kl-cor}(a), Lemma \ref{lastbit-lem}(a), and the definition of $P^\pm_{y,w}$.
% This suffices to establish part (ii) when X is A or B. In turn, the proof of Theorem \ref{recovery-thm} shows that 
% \[ \left\{ h^+_{x,y;z} + h^-_{x,y;z} : x,y,z \in W \right\} =\left \{ (f_{w,x,y;z})^2 : w,x,y,z \in W' \right\}\]
%which  proves part (ii) when X is  C or D, in light of   Lemma \ref{bal-lem}(b) and Proposition \ref{rest-vv}.
\end{proof}

In the next statement and for the duration of this section, we fix 
  an arbitrary Coxeter system  $(W,S)$ with an $S$-preserving involution $* \in \Aut(W)$, and we let $S'\subset S$ and $S'' = S\setminus S'$ be (possibly empty) sets of simple generators such that 
  \ben
  \item[(i)] $S'$ and $S''$ are each preserved by $*$;  
\item[(ii)]  Every $s' \in S'$ commutes with every $s'' \in S'' $. 
\een
We write
$W' = \langle S' \rangle $ and $ W'' = \langle S'' \rangle$
%\[W' = \langle S' \rangle \qquand W'' = \langle S'' \rangle \]
 for the  subgroups  generated by $S'$ and $S''$, and let $\I' =W' \cap \I$ and $\I'' =W''\cap\I$.
 
\begin{lemma}\label{split-lem} 
For each $w \in W$ there are unique elements in $W'$ and $W''$, which we denote  $w'$ and $w''$ respectively, such that $w=w'w'' = w''w'$. This decomposition has the following properties:  \ben
%\item[(a)] For each $w \in W$ there are unique elements in $W'$ and $W''$, which we denote  $w'$ and $w''$ respectively, such that $w=w'w'' = w''w'$. Moreover, $w \in \I$ if and only if $w' \in \I'$ and $ w'' \in \I''$, and if $y \in W$ then $y\leq w$ if and only if $y' \leq w'$ and $ y'' \leq w''$.

\item[(a)] If $w \in W$ then $w \in \I$ if and only if $w' \in \I'$ and $ w'' \in \I''.$
%$, and if $y \in W$ then $y\leq w$ if and only if $y' \leq w'$ and $ y'' \leq w''$.

\item[(b)] For all $w,x,y,z \in W$ we have 
$P_{y,w}  = P_{y',w'} P_{y'',w''}$ and $ h_{x,y;z} = h_{x',y';z'} h_{x'',y'';z''}.$

\item[(c)] For all $x \in W$ and $w,y,z \in \I$ we have 
\[ \Psig_{y,w}  = \Psig_{y',w'} \Psig_{y'',w''}\quand  \wt h_{x,y;z} = \wt h_{x',y';z'} \wt h_{x'',y'';z''} \quand h^\sigma_{x,y;z} = h^\sigma_{x',y';z'} h^\sigma_{x'',y'';z''}.\]

\een
\end{lemma}

\begin{remark}
In part (b), we  identify  $P_{y',w'}$ and $P_{y'',w''}$ with Kazhdan-Lusztig polynomials of the Coxeter systems $(W',S')$ and $(W'',S'')$. Similar identifications apply to the structure constants $h_{x',y';z'}$ and $h_{x'',y'';z''}$. 
In part (c), likewise, we  identify $\Psig_{y',w'}$, $\wt h_{x',y';z'}$, $h^\sigma_{x',y';z'}$ and $\Psig_{y'',w''}$, $\wt h_{x'',y'';z''}$, $h^\sigma_{x'',y'';z''}$ with polynomials attached to the triples $(W',S',*)$ and $(W'',S'',*)$. Note that this makes sense since $*$ restricts to an involution of $W'$ and of $W''$ which preserves $S'$ and $S''$.
\end{remark}

\begin{proof}
The first assertion and part (a)  follow from basic group theory and properties of the Bruhat order (see \cite[Exercise 2.3]{CCG}). 
Since in the Hecke algebras $\cH_q$ and $\cH_{q^2}$ we have $t_w = t_{w'}t_{w''}$ and $T_w = T_{w'}T_{w''}$ for all $w \in W$, parts (b) and (c) follow as consequences of the uniqueness specified in Theorem-Definitions \ref{kl-thmdef} and \ref{lv-thmdef3}.
\end{proof}

%
%\begin{lemma} \label{split-lem2}
%In the notation of Lemma \ref{split-lem}, for all $x \in W$ and $y,z,w \in \I$ the following holds:
%\ben
%\item[(a)] $P^+_{y,w} = P^+_{y',w'} P^+_{y'',w''} + P^-_{y',w'} P^-_{y',w'}$ and $P^-_{y,w} = P^+_{y',w'} P^-_{y'',w''} + P^-_{y',w'} P^+_{y',w'}$. 
%
%\item[(b)] $P^\pm_{y,w} - P^\pm_{z,w} =\tfrac{1}{2} (P^\pm_{y',w'} + P^\pm_{z',w'})(P^\pm_{y'',w''} - P^\pm_{z'',w''}) +\tfrac{1}{2} (P^\pm_{y',w'} - P^\pm_{z',w'})(P^\pm_{y'',w''} + P^\pm_{z'',w''})$.
%
%\item[(c)] $h^+_{x,y;z} = h^+_{x',y';z'} h^+_{x'',y'';z''} + h^-_{x',y';z'} h^-_{x'',y'';z''} $ and $  h^-_{x,y;z} = h^+_{x',y';z'} h^-_{x'',y'';z''} + h^-_{x',y';z'} h^+_{x'',y'';z''}$.
%\een
%\end{lemma}

%\begin{lemma} If $f=f(x) \in \ZZ[x]$ is symmetric and unimodal, then the polynomials $f(x)^2 \pm f(x^2)$ are both symmetric and unimodal.
%\end{lemma}
%
%\begin{proof}
%
%\end{proof}

The following result is presumably well-known to experts, but we could not locate a reference in the literature.

%For the duration of this section,  $(W,S)$ denotes an arbitrary Coxeter system with an $S$-preserving involution $* \in \Aut(W)$, and we let $S' \subset S$ and $S'' = S\setminus S'$ be (possibly empty) subsets satisfying the hypotheses in Lemma \ref{split-lem}. We write $W' = \langle S'\rangle$ and $W'' = \langle S''\rangle$ for the subgroups these sets generate.

\begin{proposition}\label{reduction-prop} 
Suppose Property \descref{B} (respectively, \descref{D}) holds for all irreducible factors of  a  Coxeter system $(W,S)$. Then Property \descref{B} (respectively, \descref{D}) holds for $(W,S)$.
\end{proposition}

Of course, the corresponding statement for Properties \descref{A} and \descref{C} holds vacuously since these properties  hold for all Coxeter systems by \cite{EWpaper}.

\begin{proof}
Let X stand for one of the letters B or D, and assume Property X holds for all irreducible factors of $(W,S)$.
We may assume without loss of generality that the rank of $(W,S)$ is finite, since any finite set of elements of $ W$ belong to a Coxeter subgroup of $ W$ generated by a finite subset of $S$, and so we can view the polynomials $P_{y,w}$ and $h_{x,y;z}$ as attached to a finite rank Coxeter system. 

We now proceed by induction on the finite rank of $(W,S)$. If $(W,S)$ is irreducible then the proposition holds automatically. If $(W,S)$ is not irreducible, then $S'$ and $S''$ can both be chosen (taking $*$ to be trivial) to be proper subsets of $S$. In this case we may assume by induction that Property X holds for the Coxeter systems $(W',S')$ and $(W'',S'')$, since these both have rank strictly less than that of $(W,S)$.
%If X is A or C then it follows immediately from Lemma \ref{split-lem}(b) that Property X  holds for $(W,S)$.
If X $=$ D then it follows from Lemma \ref{split-lem}(c) combined with  Lemma \ref{bal-lem}(a)  that Property X  holds for $(W,S)$.
 If X $=$ B then Property X holds for $(W,S)$ since in the notation of Lemma \ref{split-lem} we have
\[ P_{y,w} - P_{z,w} =\tfrac{1}{2} (P_{y',w'} + P_{z',w'})(P_{y'',w''} - P_{z'',w''}) +\tfrac{1}{2} (P_{y',w'} - P_{z',w'})(P_{y'',w''} + P_{z'',w''}) 
\]
for all $y,z,w \in W$ with $y\leq z$, and by induction all parenthesized terms on the right hand side of this identity belong to $\NN[q]$. %, as  Property \descref{B} implies Property \descref{A}.
\end{proof}

In our second proposition, recall that when we say that ``Property X$'$ holds for $(W,S)$'' we mean that the property in question holds with respect to the Coxeter system $(W,S)$ for all choices of $S$-preserving involution $* \in \Aut(W)$.

 \begin{proposition} \label{reductionprime-prop}
 Let X stand for one of the letters A, B, C, or D, and let $(W,S)$ be a Coxeter system with an $S$-preserving involution $* \in \Aut(W)$. If Properties X and X$'$ hold for all irreducible factors of $(W,S)$, then Property X$'$ holds for the triple $(W,S,*)$.  
 \end{proposition}
 
 \begin{proof}
 As in the proof of Proposition \ref{reduction-prop}, we  proceed by induction on the  rank of $(W,S)$, which we may assume to be finite, supposing Properties X and X$'$ hold with respect to any choice of involution for all irreducible factors of our Coxeter system.
 
If $S'$ and $S''$ cannot  both be chosen to  be proper subsets of $S$, then either $(W,S)$ is irreducible, or there are disjoint subsets $J',J'' \subset S$ with $S = J' \cup J''$ such that $\{ s^* : s \in J'\} = J''$ and such that the Coxeter systems  $(W_{J'},J')$ and $(W_{J''},J'')$ are both irreducible, where $W_{J'} = \langle J'\rangle$ and $W_{J''} = \langle J''\rangle$.
In the first case Property X$'$ holds for the triple $(W,S,*)$ by hypothesis. In the second case,  $W_{J'}\cong W_{J''}$ and we may identify the triple $(W,S,*)$  with a Coxeter system with involution of the form in Proposition \ref{mod-prop}. In this situation, it follows by Proposition \ref{reduction-prop} that Property X holds for the Coxeter system $(W_{J'},J')$, and so it follows in turn by Lemma \ref{recovery-thm2} that Property X$'$ holds for $(W,S,*)$.

On the other hand suppose $S'$ and $S''$ can both be chosen to be proper subsets of $S$. 
Let $x \in W$ and $y,z,w \in \I$ and observe that in the notation of Lemma \ref{split-lem} the following identities   hold:
\begin{itemize}
\item $
P^\pm_{y,w} = P^+_{y',w'} P^\pm_{y'',w''} + P^-_{y',w'} P^\mp_{y',w'} $

\item $P^\pm_{y,w} - P^\pm_{z,w} =\tfrac{1}{2} (P^\pm_{y',w'} + P^\pm_{z',w'})(P^\pm_{y'',w''} - P^\pm_{z'',w''}) +\tfrac{1}{2} (P^\pm_{y',w'} - P^\pm_{z',w'})(P^\pm_{y'',w''} + P^\pm_{z'',w''}).
$

\item $h^\pm_{x,y;z} = h^+_{x',y';z'} h^\pm_{x'',y'';z''} + h^-_{x',y';z'} h^\mp_{x'',y'';z''}
$

\end{itemize}
As we may assume by induction Property X$'$ holds for $(W',S')$ and $(W'',S'')$,
these identities (together Lemma \ref{bal-lem} and with Corollary \ref{vv-cor}) imply that Property X$'$ holds for $(W,S,*)$.
 \end{proof}

\section{Computations for finite dihedral Coxeter systems}\label{dihedral-sect}

Fix a positive integer $m \in \{3,4,5,\dots\}$ and suppose $(W,S)$ is the finite Coxeter system of type $I_2(m)$. (We require $m\geq 3$ so that $(W,S)$ is irreducible.) We take $S = \{ s,t\}$ to be a set with two elements, and define \[W = \langle s,t : s^2=t^2 = (st)^m = 1 \rangle\]
as the dihedral group of order $2m$.
It is well-known that $P_{y,w} = 1$ for all $ y,w \in W$ with $y\leq w$ (see \cite[\S4.2]{Fokko}), and we prove here the analogous result that  in the finite dihedral case, for any choice of $S$-preserving involution $* \in \Aut(W)$ one has likewise $\Psig_{y,w}=1$ for all $y,w \in \I$ with $y \leq w$. The same   statement holds in the infinite dihedral case by \cite[Proposition 3.8]{EM1}, and so  we are able to deduce here that Properties \descref{A$'$} and \descref{B$'$} hold for all Coxeter systems of rank two.
\begin{remark}
%Our treatment of type $I_2(m)$ here is incomplete.
Du Cloux has derived explicit formulas in the dihedral for the structure constants $h_{x,y;z}$; see \cite[Propositions 4.4 and 4.6]{Fokko}.
We imagine that similar formulas can be derived and used to show that Properties \descref{C$'$} and \descref{D$'$} for  dihedral Coxeter systems, but the calculations necessary for this  appear significantly more involved, and we do not undertake them here.
\end{remark}

To  denote the elements of the dihedral group $W$, we %adopt the following notation.
define for positive integers $i$  
\[  {[s,i)} = \underbrace{ststs\cdots}_{i\text{ factors}} 
\qquand
 {[t,i)} =  \underbrace{tstst\cdots}_{i\text{ factors}}. \]
% and also
% \[ 
% (i,s]=\underbrace{\cdots ststs}_{i\text{ factors}}
% \qquand
%(i,t] = \underbrace{\cdots tstst}_{i\text{ factors}}
%.
%\]
%The distinct elements of 
%$W$ are then 
%\[ 1,\qquad [s,1),\dots,[s,m-1),\qquad [t,1),\dots,[t,m-1),\qquand [s,m)= [t,m).\]
%%or alternatively
%%\[ 1,\qquad (1,s],\dots,(m-1,s],\qquad (1,t],\dots,(m-1,t],\qquand (m,s] = (m,t].\]
%The elements just listed are all reduced expressions, and 
There exist exactly two $S$-preserving involution $*$ of $W$: either $*$ is the identity automorphism or $*$ is the automorphism interchanging $s$ and $t$.
 If $*$ is trivial, then $\I$ consists of the identity, the longest element, and  all elements of $W$ of odd length, 
 i.e., 
\[ 1,\qquad [s,1),\ [s,3),\ [s,5),\ \dots \qquad  [t,1),\ [t,3),\ [t,5),\ \dots\qquand [s,m)=[t,m).
\] 
In the nontrivial case $\I$ consists of the longest element and all elements of even length, i.e., 
\[ 1,\qquad [s,2),\ [s,4),\ [s,6),\ \dots \qquad [t,2),\ [t,4),\ [t,6),\ \dots \qquand [s,m)=[t,m).
\]
%Before proving the main result of this section, we state three short lemmas in succession. 
%assume $(W,S)$ is of dihedral type $I_2(m)$, with $m \in \{3,4,\dots\}$ finite, and that 
Fix an arbitrary choice of $S$-preserving involution  $* \in \Aut(W)$
and
 write $w_0 =[s,m)=[t,m)$ for the longest element in $W$. 
 Every $w \in W$ has a unique reduced expression except  $w_0$, which has exactly two reduced expressions given by $ststs\cdots$ and $tstst\cdots$ (each with $m$ factors).
The Bruhat order on $W$ has the simple description that $y< w$ if and only if $\ell(y)  < \ell(w)$.

We note two lemmas before stating our main result.
%
%
%\begin{lemma}\label{dih-red-lem} Every $w \in W$ has a unique reduced expression except $w_0$, which has exactly two reduced expressions given by $ststs\cdots$ and $tstst\cdots$ (each with $m$ factors).
%\end{lemma}
%
%The proof of the preceding lemma is a simple exercise which we omit.
%\begin{proof}
%Certainly  $w_0$ has at least two reduced expressions and every other element has at least one. As there are only $2m+1$ distinct (possibly empty) expressions of length at most $m$ involving $s$ and $t$ without equal adjacent letters, and since $|W| = 2m$, we may replace ``at least'' in the previous sentence by ``exactly.''
%\end{proof}

\begin{lemma}\label{commute-lem}
Suppose $r \in S$ and $w \in \I$ with $rw= wr^*$. Then $m$ is odd or $*$ is trivial, such that:
\ben
\item[(a)] If $m$ is odd  and $*$ is trivial then  $w \in \{1,r\}$.

\item[(b)] If $m$ is odd and $*$ is nontrivial then $w \in \{ w_0, rw_0\}$.

\item[(c)] If $m$ is even and $*$ is trivial and $w \in \{w_0,rw_0\}$.

%\item[(d)] The case that $m$ is even and $*$ is nontrivial cannot occur.
\een
\end{lemma}

\begin{proof} Since $rw=wr^*$ if and only if $rw' = w'r^*$ where $w' = r\act w$, we may assume $rw > w$.
If $\ell(w) = 0$ then $sw=ws^*$ if and only if $s=s^*$. 
If $0 < \ell(w) < m-1$ then it follows from the previous lemma that $rw \neq wr^*$.
It remains only to consider the case when $\ell(w) = m-1$ (since when $\ell(w) = m$ it cannot hold that $rw>w$).
In this situation $rw=wr^*$ if and only if $w_0 = rw = r(rw)r^* = rw_0r^*$. One checks that this holds precisely when $m=\ell(w_0)$ is odd and $*$ is nontrivial or $m$ is even and $*$ is trivial.
\end{proof}

%Recall the definition of the polynomials $\(R^\sigma_{y,w}\)_{y,w \in \I}$ from Section \ref{algo-sect}. These polynomials have a particularly simply form for dihedral Coxeter systems, which will afford a proof of an even simpler formula for the polynomials $\Psig_{y,w}$.
%
%\begin{lemma}
%%Suppose $(W,S)$ is of dihedral type $I_2(m)$, with $m \in \{3,4,\dots\}$ finite. Let $* \in \Aut(W)$ be either of the two $S$-preserving involutions, and 
%Suppose $y,w \in \I$ such that $y<w$, and let $d = \ell(w) - \ell(y)$.
%\ben
%\item[(a)] If $d$ is even then $R^\sigma_{y,w} = q^d - 2q^{d-2} + 2q^{d-4} - \dots \mp 2q^2 \pm 1$.
%\item[(b)] If $d$ is odd then $R^\sigma_{y,w} = q^d - 2q^{d-2} + 2q^{d-4} - \dots \mp 2q \pm 1$.
%\een
%\end{lemma}
%
%\begin{remark} In other words, $R^\sigma_{y,w}$ is the monic polynomial in $q$ whose constant term is $(-1)^{\lceil d/2\rceil}$,
%whose  nonzero terms occur in degrees $\{0 \} \cup \{ d-2i : i =0,1,2,\dots, \lfloor \frac{d}{2}\rfloor\}$, and  whose non-constant non-leading nonzero coefficients (corresponding to $q^{d-2}, q^{d-4},\dots $) are given in order of decreasing degree by $-2,2,-2,2,\dots$.
%For example, if $\d \in \{1,2\}$ then $R^\sigma_{y,w} = q^d- 1$  and if $d = 6$ then $R^\sigma_{y,w} = q^6 - 2q^4 + 2q^2 -1$.
%\end{remark}
%
%\begin{proof}
%The proof is by induction on $\ell(w)$, using the recurrence for $R^\sigma_{y,w}$ in Proposition \ref{rprop}.
%\end{proof}

\begin{lemma}\label{mu1-lem}
Suppose $y,w \in \I$ and $\ell(w) - \ell(y) = 1$. Then  $m$ is odd or $*$ is trivial, such that:
\ben
\item[(a)] If $m$ is odd and $*$ is trivial then $y=1$ and $w \in S$.

\item[(b)] If $m$ is odd and $*$ is nontrivial then $y \in \{sw_0, tw_0\}$ and $w = w_0$.

\item[(c)] If $m$ is even and $*$ is trivial then  $y \in \{sw_0, tw_0\}$ and $w = w_0$, or $y = 1$ and $w \in S$.

%\item[(d)] The case that $m$ is even and $*$ is nontrivial cannot occur.
\een
\end{lemma}

\begin{proof}
The claims here follow by inspecting the lists of elements in $\I$ given before Lemma \ref{commute-lem}, noting that the elements $[s,i)$ and $[t,i)$ have length $i$ when $i\leq m$.
\end{proof}

We now have the main result of this section.
Despite the simplicity of this statement, we know of no easier proof than the following somewhat lengthy inductive argument using Corollary \ref{main-recurrence}.

\begin{theorem}\label{dihedral-AB} Suppose $(W,S)$ is of dihedral type $I_2(m)$, with $3 \leq m < \infty$. Let $* \in \Aut(W)$ be either  $S$-preserving involution. 
Then $\Psig_{y,w} = 1$ for all $y,w \in \I$ with $y \leq w$. 
\end{theorem}

\begin{proof}
Let $y,w \in \I$ such that $y \leq w$.
If $w  = 1$ then $y\leq w$ implies $y=w$ so $\Psig_{y,w} =1$ as desired. If $\ell(w) \in \{1,2\}$, then $w = r \act 1$ for some $r \in S$, in which case $y \leq w$ if and only if $y \in \{1,w\}$, whence $\Psig_{y,w}  = \Psig_{r\act y,w}= 1$ by the first part of Corollary \ref{main-recurrence}. 

For the remainder of this proof we assume that $\ell(w) \geq 3$.
We may assume that $y<w$ since $\Psig_{w,w} = 1$, and may take as an inductive hypothesis that
 $\Psig_{y',w'} = 1$ when $ w' < w$ or when $w=w'$ and $y'>y$. Let $r \in \Des(w)$ and set $w' = r\act w$. If $r \notin \Des(y)$ then $\Psig_{y,w} = \Psig_{r\act y,w} = 1$ by   hypothesis, so assume $r \in \Des(y)$. This implies that $y\neq 1$, and that $r \act y \leq w'$.

Suppose $y \not \leq w'$. Then $\ell(y) = \ell(w')$, so
the only element $z \in \I$ with $y \leq z < w$ is $z=y$,
and 
 the second part of Corollary \ref{main-recurrence} becomes 
\[  (q+1)^c \Psig_{y,w} = (q+1)^d - v^{\ell(w)-\ell(y)+c} \cdot \msig(y\xrightarrow{r}w')\] where $c = \delta_{rw,wr^*}$ and $d =\delta_{ry,yr^*}$.
To express $\msig(y\xrightarrow{r}w')$ more simply, we note that since  $\ell(y) = \ell(w')$, we have
\[\nu^\sigma(y,w') =  \mu^\sigma(y,x)\mu^\sigma(x,w') = 0\qquad\text{for all $x \in \I$,}\] and also \[ \delta_{ry,yr^*} \mu^\sigma(ry,w') =\delta_{ry,yr^*}\qquand \delta_{rw',w'r^*} \mu^\sigma(y,rw')  = \delta_{rw,wr^*} \mu^\sigma(y,w).\]
Thus, by the definition \eqref{msig-def}, our previous equation becomes
\[  (q+1)^c \Psig_{y,w} =
 (q+1)^d - q (d - c\cdot \mu^\sigma(y,w)).
\] 
If $c=0$ then this reduces to the formula $\Psig_{y,w} = (q+1)^d - dq $ which is equal to 1 for all $d \in \{0,1\}$. If $c=1$ then $\ell(y) = \ell(w') = \ell(w)-1$ so  $\mu^\sigma(y,w) $ is the constant coefficient of $\Psig_{y,w}$ and therefore equal to 1. In this case we must have $d=0$ since (using Lemma \ref{commute-lem})  the only element $x \in \I$ with $rx=xr^*$ and $\ell(x) = \ell(w) - 1$ is $w'$ which by assumption is distinct from $y$. Thus if $c=1$ then $d=0$ and our 
 equation becomes $(q+1)\Psig_{y,w} = q+1$ so $\Psig_{y,w} = 1$ again as desired.

From now on we assume $y \leq w' \leq w$. Since $r \in \Des(y) \setminus \Des(w')$, we must actually have $y < w'$. Further, since $y \neq 1$ and $w'\neq w_0$, it follows from Lemma \ref{mu1-lem} that $\ell(w') - \ell(y) \geq 2$. 
Continuing, by the second part of Corollary \ref{main-recurrence} and our inductive hypothesis, we have
\[
 (q+1)^c\Psig_{y,w} = q^2+1- \sum_{\substack{ z \in \I; \hs  rz<z \\  y\leq z < w}} v^{\ell(w) - \ell(z)+c} \msig(z\xrightarrow{r}w')
 \]
 where $c = \delta_{rw,wr^*}$. (There are no $d$'s here because $(q+1)^d +q(q-d) = q^2+1$ for all $d \in \{0,1\}$.)
We wish to replace the right  hand side of this equation with a more elementary expression.
To this end, suppose $z \in \I$ such that $rz<z$ and $y\leq z <w$. 
We make the following observations:
\ben
\item[(a)] $\mu^\sigma(z,w') = 0$. This follows because, by hypothesis, $\mu^\sigma(z,w')$ is 1 if $\ell(w') - \ell(z) = 1$ and is 0 otherwise. We cannot have $\ell(w') - \ell(z) = 1$ by Lemma \ref{mu1-lem} since $z\neq 1$ and $w' \neq w_0$.

\item[(b)] By definition and inductive hypothesis, $\nu^\sigma(z,w') = \begin{cases} 1&\text{if }\ell(w') - \ell(z) = 2 \\ 0&\text{otherwise}.\end{cases} $ %This follows by the definition of $\nu^\sigma$ and our inductive hypothesis.

\item[(c)] $\delta_{rz,zr^*} \mu^\sigma(rz,w') =0 $. This follows as $\mu^\sigma(rz,w') =0$ unless $\ell(w') - \ell(rz)  = 1$, which by Lemma \ref{mu1-lem} occurs only if $rz = 1$ and $w' \in S$  (since $w' \neq w_0$). By assumption, however, we have $\ell(w') \geq \ell(y) + 2 \geq 3$.

\item[(d)] $\delta_{rw',w'r^*} \mu^\sigma(z,rw')= c\cdot \mu^\sigma(z,w)$ by definition.

\item[(e)] $\mu^\sigma(z,x)\mu^\sigma(x,w') = 0$ for all $x \in \I$ with $r \in \Des(x)$. This follows as the product can only be nonzero if $z <x < w'$, in which case by hypothesis the product is 1 if and only if $\ell(x) = \ell(z) + 1 = \ell(w') - 1$ and is 0 otherwise. If $\ell(x) = \ell(z) + 1$, however, then $x \neq 1$, so $\ell(x) \neq \ell(w') - 1$ as $w' \neq w_0$, by Lemma \ref{mu1-lem}.

\een
In consequence of (a), we deduce that $\msig(z\xrightarrow{r}w') =0$ if  $\ell(w') - \ell(z)$ is odd, and in consequence of (b)-(e), we deduce that if $\ell(w') - \ell(z)$ is even then
\[ \msig(z\xrightarrow{r}w') = \nu^\sigma(z,w')  - c\cdot \mu^\sigma(z,w).\]
Thus, noting that $\ell(w) + c = \ell(w') + 2$, we have
\be\label{mess} (q+1)^c \Psig_{y,w} = q^2+1 - \(\sum_{z}  v^{\ell(w')-\ell(z) + 2} \cdot  \nu^\sigma(z,w')  \)+ \( \sum_z v^{\ell(w')-\ell(z) + 2} \cdot c\cdot  \mu^\sigma(z,w) \)\ee
where both sums are over $z \in \I$ with $rz<z$ and $y\leq z < w$ and $\ell(w') -\ell(z)$ even.
Recall that $\ell(w') - \ell(y) \geq 2$ and that $\ell(y) \geq 1$. 
From this and  the  description of the elements of $\I$, we  note two additional observations:
\begin{itemize}
\item
There exists exactly one element $z \in\I$ with $y \leq z <w$ and $rz<z$ and  $\ell(w') - \ell(z) $ even and $\nu^\sigma(z,w') \neq 0$. This is the element $z=r' \act w'$ where $r' \in \Des(w')\subset S$ is the generator distinct from $r \in S$, for which $\ell(w') - \ell(z) = 2$ and $\nu^\sigma(z,w') = 1$ by claim (b) above.
It follows that the first parenthesized sum in \eqref{mess} is equal to $q^2$.

\item 
If $c=1$ then by Lemma \ref{commute-lem} we must have $w=w_0$, since $\ell(w) \geq 3$ and $r \in \Des(w)$. In this case there exists exactly one element $z \in \I$ with $y < z <w$ (note that we exclude the case $y=z$) and $rz<z$ and $\ell(w') - \ell(z)$ even and $\mu^\sigma(z,w) \neq 0$. Namely, this element $z$ is given by the unique twisted involution of length $m - 1$ distinct from $w' = rw$. This element has $\ell(w') - \ell(z) = 0$ and $\mu^\sigma(z,w) = 1$, by inductive hypothesis. 
It follows that the second parenthesized sum in \eqref{mess} is equal to \[c\cdot q\ +\ c \cdot v^{\ell(w)-\ell(y) + 1} \cdot \mu^\sigma(y,w).\] 
The second term here corresponds to the summand indexed by $z=y$. Such a summand occurs if and only if $\ell(w') -\ell(y)$ is even, but our expression accounts for this circumstance because if $\ell(w') - \ell(y)$ is odd and $c \neq 0$ then nevertheless  $\mu^\sigma(y,w) = 0$, as $\ell(w) - \ell(y)$ would then not be odd.

\end{itemize}
Substituting these facts into \eqref{mess} gives 
\be\label{ohdear} (q+1)^c \Psig_{y,w} = 1\ +\ c \cdot q\ +\  c \cdot v^{\ell(w)-\ell(y)\ +\ 1} \cdot \mu^\sigma(y,w).\ee
If $c=0$ then it follows immediately that $\Psig_{y,w} = 1$. 
Suppose $c=1$. If $\ell(w) - \ell(y)$ is even then $\mu^\sigma(y,w) = 0$ so the preceding  equation becomes $(q+1) \Psig_{y,w} = q+1$ and we get likewise $\Psig_{y,w} = 1$.
Assume therefore that $\ell(w) - \ell(y)$ is odd. Define 
\[\mu_n = \mu^\sigma(y,w) \qquand n = \tfrac{\ell(w)-\ell(y) - 1}{2}\] so that by definition $\Psig_{y,w} = \mu_n q^n + \mu_{n-1} q^{n-1}+\dots +\mu_0$ for some integers $\mu_0,\dots,\mu_{n-1}$. In this notation, our  equation \eqref{ohdear} becomes
\[ (q+1)(\mu_n q^n + \mu_{n-1} q^{n-1}+\dots +\mu_0) = 1 + q + q^{n+1}\mu_n.\]
As the left hand side  is equal to $ \mu_nq^{n+1} +\sum_{i=1}^n (\mu_i + \mu_{i-1} )q^n + \mu_0$, equating coefficients of $q^i$ gives $\mu_0 = 1$ and $\mu_0 + \mu_1 = 1$ and $\mu_i + \mu_{i-1} = 0$ for $i=2,3,\dots,n$. The only solution to this system of equations is to set $\mu_0 = 1$ and $\mu_1=\mu_2=\dots=\mu_n = 0$; hence even in this final case we get $\Psig_{y,w} =1$ as desired.
\end{proof}

It follows that when $(W,S)$ is a finite dihedral Coxeter system, the polynomials  $P^-_{y,w} $ are all zero for $y,w \in \I$, while the polynomials  $P^+_{y,w}$ are 0 or 1 according to whether $y \not\leq w$ or $y \leq w$.  
We thus are left with the following corollary.

\begin{corollary} Properties \descref{A$'$} and \descref{B$'$} hold for all Coxeter systems of rank two.
\end{corollary}

\begin{proof}
This follows from  Theorem \ref{dihedral-AB} and the preceding discussion (which covers the finite irreducible dihedral case),   \cite[Theorem 3.13]{EM1} (which covers the infinite dihedral case), and Theorem \ref{x-thm} (which covers type $A_1 \times A_1$).
\end{proof}

\appendix

\begin{table}
\caption{Irreducible finite Coxeter systems with involution; see  Section \ref{reduction-sect}}
\label{types-tbl}
\begin{center}
\begin{tabular}{l  |  l  |  l}
\hline
Name & Coxeter diagram for $(W,S)$ & Involution $* \in \Aut(W)$ \\ 
\hline
$A_n$ ($n\geq 1$) 
& 
$ \xy<0.0cm,-0.0cm> \xymatrix@R=-0.0cm@C=.5cm{
*{s_1} \ar  @{-} [r]   & 
*{s_2} \ar @{-} [r] &
*{\ \cdots\ } \ar @{-} [r]  &
*{s_n}
}\endxy$ 
& 
Identity \\
$^2A_n$ ($n\geq 2$) & & Diagram  $s_i \mapsto s_{n+1-i}$
\\ & \\
\hline
$BC_n$ ($n\geq 3$) 
&  
$\xy<0.0cm,-0.0cm> \xymatrix@R=-0.0cm@C=.5cm{
*{s_1} \ar  @{-} [r]^{4}   & 
*{s_2} \ar @{-} [r] &
*{\ \cdots\ } \ar @{-} [r]  &
%*{s_{n-1}} \ar @{-} [r]  &
*{s_n}
}\endxy$
&
Identity 
\\ &\\
\hline
$D_n$ ($n\geq 4$) 
& 
$\xy<0.0cm,-0.0cm> \xymatrix@R=0.2cm@C=.5cm{
*{s_{1}}\ar  @{-} [dr]  && \\ 
  & 
*{s_3} \ar @{-} [r]  &
*{\ \cdots\ } \ar @{-} [r]  &
*{s_n}
\\
*{s_2} \ar  @{-} [ur] 
}\endxy$ 
&
Identity
\\
$^2D_n$ ($n\geq 4$) & & Diagram  $\begin{cases} s_1 \leftrightarrow s_{2} \\  s_i \mapsto s_i \text{ ($i\geq 3$)} \end{cases}$
\\ &\\
\hline
$E_6$ 
& 
$ \xy<0.0cm,-0.0cm> \xymatrix@R=0.4cm@C=.4cm{
&&*{s_{2}}&\\
*{s_1} \ar  @{-} [r]   & 
*{s_3} \ar @{-} [r] &
*{s_{4}} \ar @{-} [u]  \ar @{-} [r]  &
*{s_5} \ar @{-} [r]  &
*{s_6} 
}\endxy
$
& Identity  \\
$^2E_6$ & & Diagram  $\begin{cases} s_1 \leftrightarrow s_6 \\ s_3 \leftrightarrow s_5 \\ s_i \mapsto s_i \text{ ($i=2,4$)}\end{cases}$
\\ &\\
\hline 
$E_7$ 
&
$ \xy<0.0cm,-0.0cm> \xymatrix@R=0.3cm@C=.3cm{
&&*{s_{2}}&\\
*{s_1} \ar  @{-} [r]   & 
*{s_3} \ar @{-} [r] &
*{s_{4}} \ar @{-} [u]  \ar @{-} [r]  &
*{s_5} \ar @{-} [r]  &
*{s_6} \ar @{-} [r]  &
*{s_7} 
}\endxy
$ 
& Identity  
\\ &\\
\hline 
$E_8$ 
&
$ \xy<0.0cm,-0.0cm> \xymatrix@R=0.3cm@C=.3cm{
&&*{s_{2}}&\\
*{s_1} \ar  @{-} [r]   & 
*{s_3} \ar @{-} [r] &
*{s_{4}} \ar @{-} [u]  \ar @{-} [r]  &
*{s_5} \ar @{-} [r]  &
*{s_6} \ar @{-} [r]  &
*{s_7} \ar @{-} [r]  &
*{s_8}
}\endxy
$ 
& Identity  
\\ &\\
\hline 
$F_4$ 
& 
$ \xy<0.0cm,-0.0cm> \xymatrix@R=-0.0cm@C=.5cm{
*{s_1} \ar  @{-} [r]   & 
*{s_2} \ar  @{-}^{4} [r]   & 
*{s_3} \ar  @{-} [r]   & 
*{s_4}
}\endxy
$
& Identity  
\\
$^2F_4$ & & Diagram  $s_i \mapsto s_{5-i}$ 
\\ &\\
\hline
$H_3$ &
$\xy<0.0cm,-0.0cm> \xymatrix@R=-0.0cm@C=.5cm{
*{s_1} \ar  @{-} [r]^{5}   & 
*{s_2} \ar @{-} [r] &
*{s_3}
}\endxy$ 
& Identity  
\\ &\\
\hline
$H_4$ & 
$\xy<0.0cm,-0.0cm> \xymatrix@R=-0.0cm@C=.5cm{
*{s_1} \ar  @{-} [r]^{5}   & 
*{s_2} \ar @{-} [r] &
*{s_3} \ar @{-} [r] &
*{s_4}
}\endxy$ 
& Identity  
\\ &\\
\hline
$I_2(m)$ ($m\geq 4$) 
& 
$\xy<0.0cm,-0.0cm> \xymatrix@R=-0.0cm@C=.5cm{
*{s_1} \ar  @{-} [r]^{m}   & 
*{s_2}
}\endxy$
& Identity  \\
$^2I_2(m)$ ($m\geq4$) & & Diagram  $s_i \mapsto s_{3-i}$
\\ &\\
\hline
\end{tabular}
\end{center}
All Coxeter diagrams  are labeled to coincide with the indexing conventions in {\tt Coxeter} \cite{Coxeter}.
The  types $BC_2$, $^2BC_2$, $G_2$, $^2G_2$ are omitted since they coincide with types  $I_2(m)$, $^2I_2(m)$ for $m=4,6$.
\end{table}

%\begin{table}
%\caption{Minimum nonzero coefficients in KL-type polynomials; see  Section \ref{compute-sect}}
%\label{mincoeff-P-tbl}
%\begin{center}
%{\small
%\begin{tabular}{r|rrrr}
%\hline
%Type & $P_{y,w}$ ($y,w \in \I$)  & $\Psig_{y,w}$ & $P^+_{y,w}$ & $P^-_{y,w}$ \\
%\hline
%$A_3$  & 1 & 1 & 1 & (all polynomials are zero)  \\
%$A_4$  & 1 & 1 & 1 &   (all polynomials are zero)  \\
%$A_5$ & 1 &1 &1 &1  \\
%$A_6$ & 1 &1 &1 &1  \\
%\hline
%$^2A_3$   & 1 & $-$1 & 1 &  1 \\
%$^2A_4$  & 1 & $-$1 & 1 & 1  \\
%$^2A_5$ & 1 & $-$1 & 1 & 1 \\
%$^2A_6$ & 1 & $-$2 & 1 & 1 \\
%\hline
%$BC_3$  & 1 & $-$1 & 1 &  1 \\
%$BC_4$   & 1 & $-$1 & 1 &  1 \\
%$BC_5$ &  1 & $-$3  & 1 & 1 \\
%$BC_6$ & 1& $-8$ & 1 &1  \\
%\hline
%$D_4$  & 1 & $-$2 & 1 &  1 \\
%$D_5$ & 1  & $-$3 & 1 & 1 \\
%$D_6$ & 1& $-12$ & 1&  1\\
%\hline
%$^2D_4$  & 1 & $-$1 & 1 &  1 \\
%$^2D_5$ & 1 & $-$2 & 1 & 1 \\
%$^2D_6$ &1& $-5$& 1& 1 \\
%\hline
%$E_6$ & 1 & $-10$ & 1 & 1 \\
%\hline
%$^2E_6$ & 1 & $-3$ & 1 & 1 \\
%\hline
%$F_4$  & 1 & $-$2 & 1 &  1  \\
%\hline
%$^2F_4$ & 1 & $-$1 & 1 &  1 \\
%\hline
%$H_3$  & 1 & $-$1 & 1 &  1 \\
%$H_4$  & 1 & $-$9 & 1 &  1 \\
%\hline
%$I_2(m)$, $m\geq 3$  & 1 & 1 & 1 & (all polynomials zero) \\
%\hline
%$^2I_2(m)$, $m\geq 3$   & 1 & 1 & 1 & (all polynomials zero) \\
%\hline
%\end{tabular}
%}
%\end{center}
%\end{table}

\begin{table}
\caption{Maximum nonzero coefficients in KL-type polynomials; see  Section \ref{conj-sect}}
\label{maxcoeff-P-tbl}
\begin{center}
{\small
\begin{tabular}{r|rrrrr}
\hline
Type & $P_{y,w}$ ($y,w \in \I$)  & $\Psig_{y,w}$ & $-\Psig_{y,w}$ & $P^+_{y,w}$ & $P^-_{y,w}$ \\
\hline
$A_1$, $A_2$, $A_3$  & 1 & 1 & $-1$	& 1 &  (all polynomials zero)  \\
$A_4$  & 2 & 2& $-1$	& 2 &   (all polynomials zero)  \\
$A_5$ & 4 &4 & $-1$	&4 & 1 \\
$A_6$ & 15 & 7& $-1$	& 11 & 4 \\
$A_7$ & 73 & 25 & $-1$ & 49 & 25 \\
$A_8$ & 362 & 82 & $-1$ & 222 & 140 \\
\hline
$^2A_2$   & 1 & 1& 	$-1$	& 1 &  1 \\
$^2A_3$   & 1 & 1& 	1	& 1 &  1 \\
$^2A_4$  & 2 & 1& 	1	& 1 &  1  \\
$^2A_5$ & 4 & 2 &	1	& 3 & 2 \\
$^2A_6$ & 15 & 3& 	2	& 8 & 7 \\
$^2A_7$ & 73 & 5& 	3	& 38 & 35 \\
$^2A_8$ & 460 & 12& 	6	& 232 & 228 \\
\hline
$BC_3$  & 1 & 1& 	1	& 1 &  1 \\
$BC_4$   & 5 & 3& 	1	&	 4 &  1  \\
$BC_5$ &  35 & 10& 3	& 21 & 14 \\
$BC_6$ &  454 & 48 &8	& 246 & 208 \\
\hline
$D_4$  & 4 & 3& 	2	& 3 &  2 \\
$D_5$ &  17 &8& 	3	& 11 & 6 \\
$D_6$ &  217 & 25& 12	& 121 & 96 \\
\hline
$^2D_4$  & 10 & 8 &	1& 7 &  2 \\
$^2D_5$ & 17 & 4&		2&10 & 7 \\
$^2D_6$ & 217 & 18& 	5&116 & 101 \\
\hline
$E_6$ & 581 & 54 &	10	& 293 & 288 \\
\hline
$^2E_6$ & 748 & 18& 3	& 374 & 374 \\
\hline
$F_4$  & 12 & 8& 	2	& 9 &  5 \\
\hline
$^2F_4$ & 12 & 2 &	1	& 6 &  6 \\
\hline
$H_3$  & 3 & 1 &	1	& 2 &  1 \\
$H_4$  & 5,116 & 213& 9	& 2,651 &  2,465 \\
\hline
%$I_2(m)$ ($ m\geq 4)$  & 1 & 1& $-1$	& 1 & (all polynomials zero) \\
%\hline
%$^2I_2(m)$ $(m \geq 4)$   & 1 & 1& $-1$	& 1 & (all polynomials zero) \\
%\hline
\end{tabular}
}
\end{center}
We obtained this data by running our extended version of  {\tt Coxeter}  \cite{MyCode} for the triples $(W,S,*)$ of the types listed. 
Our computations verify 
Properties \descref{A$'$} and \descref{B$'$}  in each of these types.
\end{table}

%\clearpage
\newpage

\begin{table}
\caption{Maximum nonzero coefficients in KL-type structure constants; see  Section \ref{conj-sect}}
\label{maxcoeff-h-tbl}
\begin{center}
{\footnotesize
\begin{tabular}{r|rrrrr}
\hline
Type & $\wt h_{x,y;z}$ ($x \in W$; $y,z \in \I$) & $h^\sigma_{x,y;z}$ & $-h^\sigma_{x,y;z}$   & $h^+_{x,y;z}$ & $h^-_{x,y;z}$ \\
\hline
$A_1$ & 2 & 1 & $-1$ & 1 & 1 \\
$A_2$ & 10 & 2 & $-1$ & 5 & 5\\
$A_3$  & 132 & 10 &$-1$ & 66 & 66\\
$A_4$  & 3,748 & 61 & $-1$  & 1,892 & 1,856\\
$A_5$ & 922,740 &912 & $-1$  & 461,826 & 460,914\\
$A_6$ & 179,487,027 & 20,367 & $-1$ & 89,753,697 & 89,733,330 \\
\hline
$^2A_2$ & 10 & 2 & 1 & 5 & 5 \\
$^2A_3$  & 132 & 7 &3 & 66 & 66\\
$^2A_4$ & 4,698 & 36 & 10  & 2,358 & 2,340\\
$^2A_5$ & 922,740 & 506 & 162 & 461,404 &461,336 \\
$^2A_6$ & 186,996,750 & 4,080 & 1,994 & 93,499,109 & 93,497,641 \\
\hline
$BC_2$ & 14 & 2 & 1 & 8 & 6\\ 
$BC_3$  & 905 & 28 &8 & 451 & 454 \\
$BC_4$  & 397,846 & 767 & 156 & 199,042 & 198,804 \\
$BC_5$ & 1,319,190,596 &42,248 &9,924 &659,608,306 &659,582,290  \\
\hline
$^2BC_2$ & 14 & 2 & 1 & 8 & 6 \\
\hline
$D_4$ & 42,384 & 246 & 85  & 21,226 & 21,225  \\
$D_5$ & 89,307,651 & 11,123 &3,319 & 44,652,166 & 44,655,485 \\
\hline
$^2D_4$ & 42,384 & 116 & 30 &  21,225 & 21,159 \\
$^2D_5$ & 89,307,651 &4,748 &1,538 & 44,655,112 & 44,652,539 \\
\hline
$F_4$ & 108,380,588 & 8,995 & 2,007 & 54,192,072 & 54,188,516  \\
\hline
$^2F_4$ & 108,380,588 & 2,600 & 86 & 54,191,594 & 54,188,994 \\
\hline 
$G_2$ & 22 & 2 & 2 & 12 & 10 \\
\hline
$^2G_2$ & 22 & 2 & 1 & 12 & 10 \\
\hline
$H_3$  & 15,676 & 106 &49& 7,870 & 7,806 \\
$H_4$ &  59,133,414,193,112,056 & 467,325,554 & 60,353,800 & 29,566,707,126,594,414 & 29,566,707,066,517,642\\
\hline
\end{tabular}

%\bigskip
%
%
%\begin{tabular}{r|rr}
%\hline
%Type & $h^+_{x,y;z}$  ($x \in W$; $y,z \in \I$) & $h^-_{x,y;z}$ \\
%\hline
%$A_1$ & 1 & 1 \\
%$A_2$ & 5 & 5 \\
%$A_3$  &66 & 66\\
%$A_4$  & 1,892 & 1,856 \\
%$A_5$ &461,826 &460,914 \\
%$A_6$ & 89,753,697 & 89,733,330\\
%\hline
%$^2A_2$ & 5 & 5 \\
%$^2A_3$  &66 & 66\\
%$^2A_4$ & 2,358 & 2,340 \\
%$^2A_5$ & 461,404 &461,336 \\
%$^2A_6$ & 93,499,109 & 93,497,641\\
%\hline
%$BC_2$ & 8 & 6\\
%$BC_3$  &451 & 454\\
%$BC_4$  & 199,042 & 198,804\\
%$BC_5$ &659,608,306 &659,582,290 \\
%\hline
%$D_4$ & 21,226 & 21,225 \\
%$D_5$ &44,652,166 & 44,655,485 \\
%\hline
%$^2D_4$  & 21,225 & 21,159 \\
%$^2D_5$ &44,655,112 & 44,652,539 \\
%\hline
%$F_4$ & 54,192,072 & 54,188,516 \\
%\hline
%$^2F_4$ & 54,191,594 & 54,188,994 \\
%\hline
%$H_3$  & 7,870 & 7,806\\
%$H_4$ & 29,566,707,126,594,414 & 29,566,707,066,517,642\\
%\hline
%\end{tabular}
}
\end{center}
We obtained this data by running our extended version of  {\tt Coxeter}  \cite{MyCode} for the triples $(W,S,*)$ of the types listed. 
Our computations verify 
Properties \descref{C$'$} and \descref{D$'$}  in each of these types.
\end{table}

\end{document}